\RequirePackage[l2tabu,orthodox,abort]{nag}

\documentclass[
12pt,
a4paper,
reqno,
tbtags
]
{amsart}

\usepackage[all,error]{onlyamsmath}

\usepackage[a4paper,left=3cm,right=3cm,top=3cm,bottom=3cm]{geometry}

\usepackage[T1]{fontenc}
\usepackage[utf8x]{inputenc}

\usepackage[strict=true]{csquotes}

\usepackage{graphicx}
\usepackage{amsmath}
\usepackage{amsfonts}
\usepackage{amssymb}
\usepackage{amsthm}
\usepackage{bbm}
\usepackage{color}
\usepackage{marginnote}

\usepackage{subfig}

\usepackage{enumitem}

\newenvironment{enumeratee}[1][]
  {\enumerate[label=\textnormal{(E\arabic*)}, itemsep=4pt, #1]}
  {\endenumerate}

\newenvironment{enumerateh}[1][]
  {\enumerate[align=left,
leftmargin=*,
labelsep*=1em,
topsep=6pt,
label=\textup{(H\arabic*)},
itemsep=6pt,
#1]}
  {\endenumerate}

\newenvironment{enumeratev}[1][]
  {\enumerate[align=left,
leftmargin=*,
labelsep*=1em,
label=\textup{(V\arabic*)},
#1]}
  {\endenumerate}

\makeatletter
\def\namedlabel#1#2{\begingroup
    #2%
    \def\@currentlabel{#2}%
    \phantomsection\label{#1}\endgroup
}
\makeatother

\numberwithin{equation}{section}

\theoremstyle{plain}

\newtheorem{theorem}{Theorem}[section]
\newtheorem{proposition}[theorem]{Proposition}
\newtheorem{lemma}[theorem]{Lemma}

\theoremstyle{definition}

\newtheorem{definition}[theorem]{Definition}

\newtheorem{remark}[theorem]{Remark}

\newcommand{\R}{\mathbb{R}}

\renewcommand{\P}{\mathbb{P}}

\newcommand{\Y}{\mathbb{Y}}

\newcommand{\ind}{\mathbbm{1}}

\providecommand{\cl}[1]{\mathcal{#1}}
\newcommand{\pe}{\phantom{={}}}

\renewcommand{\phi}{\varphi}
\renewcommand{\epsilon}{\varepsilon}

\newcommand{\e}{\mathrm{e}}

\newcommand{\dd}[1]{\mathop{}\!\mathrm{d}#1}
\newcommand{\ddd}{\mathrm{d}}
\newcommand{\dmu}{\,\mathrm{d}\mu}

\newcommand{\tif}{\tilde f}
\newcommand{\tio}{\tilde \Omega}
\newcommand{\tiy}{\tilde y}

\newcommand{\tit}{\tilde T}
\newcommand{\tia}{\tilde A}
\newcommand{\sem}{\lrc{T(t),\, t \geq 0}}

\newcommand{\sems}{\lrc{S(t),\, t \geq 0}}
\newcommand{\sss}{S}
\newcommand{\sst}{T}
\newcommand{\tisem}{\lrc{\tilde T(t),\, t \geq 0}}
\newcommand{\tisemy}{\lrc{\tilde T_{\Y}(t),\, t \geq 0}}

\newcommand{\lo}{L^{1}(\Omega)}

\newcommand{\wo}{W^{1}(\Omega)}

\newcommand{\linf}{L^{\infty}(\Omega)}

\newcommand{\invm}{f_{\diamond}}
\newcommand{\nuw}{\mathop{}\!{\nu(\ddd w)}}
\newcommand{\nuv}{\mathop{}\!{\nu(\ddd v)}}

\newcommand{\mes}{\mathop{}\!\ell}
\newcommand{\mess}{\mathop{}\!\ell^{*}}

\newcommand{\ltio}{L^1_{\omega}(\tio)}
\newcommand{\wtio}{W^1_{\omega}(\tio)}
\newcommand{\normtio}[1]{\normt{#1}_{\ltio}}
\newcommand{\normo}[1]{\norm{#1}_{\lo}}

\newcommand{\lino}[1]{\mathcal{L}(#1)}
\newcommand{\set}[1]{\lrc{#1}}

\newcommand{\exts}{\Y}

\newcommand{\tm}{k}
\newcommand{\tms}{k^*}

\DeclareMathOperator{\leb}{leb}
\newcommand{\lebm}{\mathrm{leb}}
\DeclareMathOperator*{\esssup}{ess\,sup\,}

\usepackage{mathtools}
\DeclarePairedDelimiter{\lrp}{\lparen}{\rparen}
\DeclarePairedDelimiter{\lrb}{\lbrack}{\rbrack}
\DeclarePairedDelimiter{\lrc}{\lbrace}{\rbrace}
\DeclarePairedDelimiter{\abs}{\lvert}{\rvert}
\DeclarePairedDelimiter{\norm}{\lVert}{\rVert}

\DeclarePairedDelimiter\ceil{\lceil}{\rceil}
\DeclarePairedDelimiter\floor{\lfloor}{\rfloor}

\makeatletter
\newcommand{\normt}{\@ifstar\@normts\@normt}
\newcommand{\@normts}[1]{%
  \left|\mkern-1.5mu\left|\mkern-1.5mu\left|
   #1
  \right|\mkern-1.5mu\right|\mkern-1.5mu\right|
}
\newcommand{\@normt}[2][]{%
  \mathopen{#1|\mkern-1.5mu#1|\mkern-1.5mu#1|}
  #2
  \mathclose{#1|\mkern-1.5mu#1|\mkern-1.5mu#1|}
}
\makeatother

\renewcommand{\leq}{\leqslant}
\renewcommand{\geq}{\geqslant}

\usepackage[many]{tcolorbox}

\newcounter{myhypo}
\renewcommand\themyhypo{(H\arabic{myhypo})}

\newtcolorbox{hypo}[1][]{
  breakable,
  enhanced,
  before={\medskip\noindent},
  after={},
  top=0pt,
  bottom=0pt,
  colback=white,
  boxrule=0pt,
  boxsep=0pt,
  left=40pt,
  right=0pt,
  leftrule=0pt,
  colframe=white,
  outer arc=0pt,
  overlay={
    \node[inner sep=0pt,anchor=west]
    at (frame.west)
    {\refstepcounter{myhypo}\themyhypo\label{#1}};
  }, 
}

\usepackage{hyperref}

\title[Lord Kelvin's method of images approach to the Rotenberg model]{Lord Kelvin's method of images approach to the Rotenberg model and its asymptotics}
\author[A. Gregosiewicz]{Adam Gregosiewicz}

\address{%
  Institute of Mathematics, Polish Academy of Sciences, ul.
  Sniadeckich 8, 00-656 Warsaw, Poland}

\address{%
  Lublin University of Technology,\ ul.~Nadbystrzycka 38A,\ 20-618 Lublin,
  Poland}

 \email{a.gregosiewicz@pollub.pl}


\thanks{This research was partly supported by Polish National Science Centre
  grant 2014/15/\ N/ST1/03110.}

\begin{document}

\begin{abstract}
We study a~mathematical model of cell populations dynamics proposed by
M.~Rotenberg~\cite{rotenberg} and investigated by
M.~Boulanouar~\cite{boulanouar}.
Here, a~cell is characterized by her maturity and speed of maturation.
The growth of cell populations is described by a~partial differential equation
with a~boundary condition.
In the first part of the paper we exploit semigroup theory approach and apply
Lord Kelvin's method of images in order to give a~new proof that the model is
well posed.
Next, we use a semi-explicit formula for the semigroup related to the model
obtained by the method of images in order to give growth estimates for the
semigroup.
The main part of the paper is devoted to the asymptotic behaviour of the
semigroup.
We formulate conditions for the asymptotic stability of the semigroup in the
case in which the average number of viable daughters per mitosis equals one.
To this end we use methods developed by K.~Pichór and
R.~Rudnicki~\cite{pichrud}.
\end{abstract}

\maketitle

\section{Introduction}
\label{sec:introduction}

In the Rotenberg model of cell populations dynamics~\cite{rotenberg} a~cell is
characterized by two variables, its maturity and speed of maturation. 
We assume that the maturity is a~real number \( x \) that belongs to the
interval \( I \coloneqq (0,1) \) and the speed of maturation \( v \) belongs to
the~set \( V \coloneqq (a,b) \), where \( a \) and \( b \) are nonnegative real
numbers such that \( a < b < +\infty \).
Growth of the~cells' population density is governed by the partial differential
equation
\begin{equation}
\label{eq:pde}
\frac{\partial f}{\partial t} = -v \frac{\partial f}{\partial x},
\end{equation}
where \( f = f(x,v,t) \) with \( t \geq 0 \) is the cells' density at
\( (x,v) \) at time \( t \).
In this model a~cell starts maturing at \( x = 0 \) and divides reaching
\( x = 1 \), and the boundary condition
\[
vf(0,v,t) = p \int_V w \tm(w,v) f(1,w,t) \dd w
\]
describes the reproduction rule.
Here \( \tm \) satisfies
\begin{equation*}
\int_V \tm(w,v) \dd v = 1
\end{equation*}
for any \( w \in V \), and \( V \ni v \mapsto \tm(w,v) \) is the probability
density of daughter velocity, conditional on \( w \) being the velocity of the
mother.
Furthermore, it is assumed that \( p \geq 0 \) is the average number of viable
daughters per mitosis.
However, see~\cite{boulanouar}, it may be also important to consider the case
where there are cells that degenerate in the sense that theirs daughters inherit
mother's velocity.
Such situation is described by the boundary condition
\[
f(0,v,t) = q f(1,v,t),
\]
where \( q \geq 0 \) is the average number of viable daughters per mitosis.
Therefore, we combine these two cases and assume that the reproduction rule is
characterized by the boundary condition
\begin{equation}
\label{eq:ogolny-bc}
vf(0,v,t) = p \int_V w \tm(w,v) f(1,w,t) \dd w + q v f(1,v,t), \qquad v \in V.
\end{equation}

It is well known, see~\cite[II.1.2]{goldstein}, that the well-posedness of the
problem~\eqref{eq:pde}-\eqref{eq:ogolny-bc} may be rephrased in terms of the
semigroups theory as follows: The problem is well-posed if and only if the
operator
\[
f \mapsto -v \frac{\partial f}{\partial x}
\]
with domain consisting of functions that are absolutely continuous with respect
to \( x \) and satisfy~\eqref{eq:ogolny-bc} is the generator of a~strongly
continuous semigroup in the space of absolutely integrable functions.

In the first part of this paper, in Section~\ref{sec:generation-theorem}, we
give a~new proof of the generation theorem of Boulanouar~\cite[Theorem 2.2,
Theorem 3.1]{boulanouar}.
To this end we use Lord Kelvin's method of images.
(For detailed introduction to the method of images see~\cite{bobrowski10a} and
references given there.
More examples may be found in \cite{bobrowski10b,bobgre,bobgremur}.)
As a~by-product we obtain a semi-explicit formula for the semigroup
\( \sst = \sem \) related to the Rotenberg model.
Moreover, this formula allows us to provide, in
Section~\ref{sec:growth-estimates}, growth estimates for the Rotenberg
semigroup.

Section~\ref{sec:asymptotic-stability} is devoted to the asymptotic behaviour of
the semigroup.
Boulanouar proved~\cite[Theorem~6.1]{boulanouar} that if \( p + q > 1 \), then
properly rescaled Rotenberg semigroup converges in the uniform topology to
a~rank one projection under some conditions on \( k \).
(For precise statement, see Theorem~\ref{thm:boulanouara}.)
We study the model in the case \( p+q \leq 1 \); this case is also of biological
interest since in multicellular organisms the number of cells does not grow in
an unrestricted way, as in the case \( p + q > 1 \).

We start by noting that in the case \( p + q = 1 \) the Rotenber semigroup is
composed of Markov operators -- this remark allows us to use the tools of the
rich theory of Markov semigroups (see for example~\cite{lasota,pichrud}).
Then, using recent results of Pichór and Rudnicki~\cite{pichrud}, we prove that,
for a fairly large class of kernels \( \tm \), there is an invariant density
\( f_{*} \) for the Rotenberg semigroup such that for all other densities
\( f \) we have
\begin{equation}
\label{as}
\lim_{t\to \infty} \norm{ T(t)f - f_* } = 0 ,
\end{equation}
in an appropriate \( L^1 \)-type norm.
Interestingly, in this case there is a direct formula connecting \( f_{*} \)
with the stationary density for the kernel \( k \).
Moreover, we show that if \( p + q < 1 \), then the operators forming the
Rotenberg semigroup converge to zero in the strong operator topology; if,
additionally, \( a > 0 \), the same is true in the operator norm.
The last two statements are reflections of the fact that in the case
\( p + q < 1 \) the cell population gradually dies out.
In the case \( a > 0 \) all parts of the population die out uniformly fast.
In the case \( a = 0 \), cells that mature very slowly survive much longer than
the remaining cells and so the population dies out non-uniformly.

Our main result, combining Theorems \ref{thm:asymptotic} and \ref{thm:asymp-spadek} may be rephrased as follows.

\begin{theorem}
\label{thm:main}
Let \( T \) be the semigroup related to the Rotenberg model.
\begin{enumerate}
  \item Suppose that \( p > 0 \) and \( p+q = 1 \).
Assume that there exists a~unique, up to an equivalence class, stationary
density for the kernel \( \tm \), that is, a~nonnegative function \( \invm \) on
\( V \) with \( \int_V \invm = 1 \), satisfying
\begin{equation*}
\invm(v) = \int_V \tm(w,v) \invm(w) \dd w
\end{equation*}
for almost every \( v \in V \).
If \( \invm \) is strictly positive almost everywhere and the function
\[
v \mapsto v^{-1} \invm(v)
\]
is integrable on \( V \), then~\eqref{as} holds with \( f_{*} \) defined as
\begin{equation*}
f_{*}(v) \coloneqq \frac{v^{-1} \invm(v)}{ \int_V w^{-1} \invm(w) \dd w },
\qquad v \in V.
\end{equation*}
  \item Suppose that \( p + q < 1 \).
Then
\begin{equation*}
\lim_{t\to \infty } \norm{ T(t)f } = 0
\end{equation*}
holds for any \( f \) that is integrable on \( I \times V \).
Moreover, if \( a > 0 \), then
\[
\lim_{t\to \infty} \norm{ T(t) } = 0.
\]
\end{enumerate}
\end{theorem}

In the last part, in Section~\ref{sec:disc-assumpt}, we discuss relations
between Boulanouar's assumptions~\cite[Theorem~6.1]{boulanouar} on \( \tm \)
with these in Theorem~\ref{thm:main}.

\section{Generation theorem}
\label{sec:generation-theorem}

As in Introduction we consider \( I \coloneqq (0,1) \) as a~measure space with
the Lebesgue measure, denoted \( \lebm \), and fix real numbers \( a, b \) such
that \( 0 \leq a < b < +\infty \).
We also let \( V \subseteq (a,b) \) and introduce a~measure \( \nu \) on
\( V \).
From the biological point of view the most interesting cases are when \( V \)
equals \( (a,b) \) or is its discrete subset (the underlying measure \( \nu \)
being the Lebesgue measure or the counting measure, respectively).
However, in our generation theorem we do not need to assume that, and we can
work in the abstract setup.
Generalizations of the Rotenberg model in the discrete case are discussed
in~\cite{banasiak15,banasiak16}.

\begin{figure}
    \centering
    \subfloat[Continuous case.]{
        \centering\includegraphics[width=0.29\textwidth]{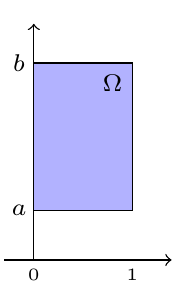}
        \label{fig:cont}}
    \qquad\qquad\qquad
    \subfloat[Discrete case.]{
        \centering\includegraphics[width=0.29\textwidth]{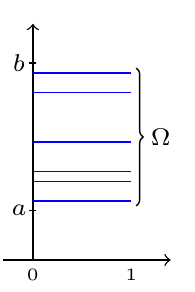}
        \label{fig:disc}}
      \caption{The set \( \Omega \).}
\label{fig:omega}
\end{figure}
We denote
\[
\Omega \coloneqq I \times V,
\]
see Figure~\ref{fig:omega}, and introduce \( \lo \) as the space of (equivalence
classes of) absolutely integrable real functions on \( \Omega \) with respect to
the product measure \( \mu = \mathrm{leb} \times \nu \).
We also denote the standard \( L^1 \)-norm by \( \norm{ \, \cdot \, }_{\lo} \).
Furthermore, we let \( \wo \) to be the space of (equivalence classes of)
functions \( f \in \lo \) satisfying:
\begin{enumerate}
  \item for almost every \( v \in V \) the function \( I \ni x \mapsto f(x,v) \)
is weakly differentiable,
  \item the function \( \Omega \ni (x,v) \mapsto v \partial_x f(x,v) \) belongs
to \( \lo \), where \( \partial_x f \) is the weak derivative of \( f \) with
respect to \( x \).
\end{enumerate}
By the Sobolev embedding theorem, if \( f \in \wo \), then for almost every
\( v \in V \) the function \( I \ni x \mapsto f(x,v) \) has a~unique
representant which is continuous up to the boundary of \( I \).
Hence, in particular we can speak about \( f(0,v) \) or \( f(1,v) \).

Let \( \tm\colon V\times V \to [0,+\infty) \) be a~nonnegative
\( (\nu \times \nu) \)-measurable real function such that
\begin{equation}
\label{eq:rozklad-k}
\int_V \tm(w,v) \nuv = 1, \qquad w \in V.
\end{equation}
Then we define the operator \( A \) in \( \lo \) by
\begin{equation}
\label{eq:gen}
Af(x,v) \coloneqq -v \partial_x f(x,v), \qquad (x,v) \in \Omega.
\end{equation}
We let the domain \( D(A) \) of \( A \) to be composed of functions
\( f \in \wo \) satisfying the boundary condition
\begin{equation}
\label{eq:bc}
v f(0,v) = p \int_V w \tm(w,v) f(1,w) \nuw + q v f(1,v)
\end{equation}
for almost every \( v \in V \), where \( p,q \) are fixed nonnegative real
numbers such that \( p + q > 0 \).
For simplicity of notation, for a~given \( v \in V \) we introduce the measure
\( \mes(\cdot,v) = \mes_{\nu}(\cdot,v) \) on \( V \) by the formula
\begin{equation}
\label{eq:mes}
\mes(\dd w,v) \coloneqq p wv^{-1} \tm(w,v) \nu(\dd w) + q \delta_v(\dd w),
\end{equation}
where \( \delta_v \) is the Dirac measure at \( v \).
Then we may rewrite~\eqref{eq:bc} in the form
\begin{equation}
\label{eq:bc-mes}
f(0,v) = \int_V f(1,w) \mes(\dd w,v).
\end{equation}

The aim of this section is to prove the following result.
\begin{theorem}
\label{thm:glowne}
The operator \( A \) generates a~strongly continuous semigroup in \( \lo \).
\end{theorem}

To prove this theorem we can use the Lord Kelvin method of images.
Indeed, formula~\eqref{eq:gen} indicates that for a~fixed \( v \in V \) the
desired semigroup should resemble a~translation semigroup.
Hence, we would like to define \( \sst = \sem \) in \( \lo \) by
\begin{equation}
\label{eq:sem}
T(t) f(x,v) = \tif(x-tv,v), \qquad t \geq 0,\ (x,v) \in \Omega,\ f \in \lo,
\end{equation}
where \( \tif \) is a~function defined on
\begin{equation}
\label{eq:tio-def}
\tio \coloneqq J \times V
\end{equation}
for \( J \coloneqq (-\infty,1) \).
Since \( T(0)f \) equals \( f \), it follows that \( \tif \) must be an
extension of \( f \).
Moreover, because every semigroup leaves the domain of its generator invariant,
given \( f \in \lo \) we are looking for \( \tif\colon \tio \to \R \) such that
\begin{enumeratee}
  \item\label{item:e1} the restriction of \( \tif \)
to \( \Omega \) equals \( f \), that is, \( \tif_{|\Omega} = f \),
  \item\label{item:e2} if \( f \in D(A) \), then \( T(t) f \) given
by~\eqref{eq:sem} belongs to \( D(A) \) for \( t \geq 0 \).
\end{enumeratee}
Construction of such extension \( \tif \) of \( f \) is the main part of the
method of images.

\begin{lemma}
\label{lem:uniq}
Given \( f \in D(A) \), if there exists \( \tif\colon \tio \to \R \)
satisfying~\ref{item:e1} and~\ref{item:e2}, then it is uniquely determined.
\end{lemma}

\begin{proof}
Let \( f \in D(A) \).
Condition \ref{item:e2} implies in particular that \( \tif \) must be chosen in
such a~way that \( T(t) f \) given by~\eqref{eq:sem} satisfies the boundary
condition~\eqref{eq:bc-mes}.
Hence, we must have
\begin{equation*}
\tif(-tv,v) = \int_V \tif(1-tw,w) \mes(\dd w,v)
\end{equation*}
for \( t \geq 0 \) and almost every \( v \in V \).
If we denote \( x = -tv \), this may be rewritten as
\begin{equation}
\label{eq:ext}
\tif(x,v) = \int_V \tif\lrp{1 + xwv^{-1},w} \mes(\dd w,v),
\qquad x \leq 0,\ v \in V.
\end{equation}
For a~positive integer \( i \) we set
\begin{equation}
\label{eq:omegi}
\Omega_i \coloneqq \lrc[\big]{ (x,v) \in \R^2\colon v \in V,\ -ivb^{-1} < x \leq
  -(i-1)vb^{-1} },
\end{equation}
see Figure~\ref{fig:omegi}.
\begin{figure}[h]
\centering
\includegraphics[scale=1.75]{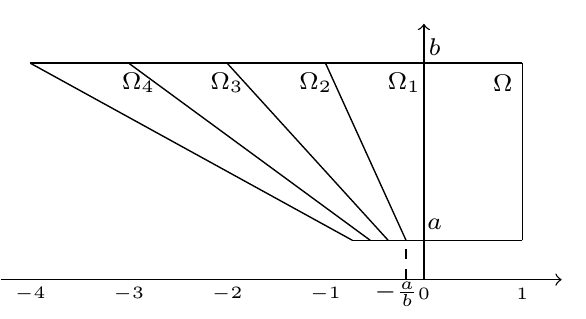}
\caption{\( \tio \) is the union of \( \Omega_i \)'s.}
\label{fig:omegi}
\end{figure}

For \( w \in V \) and \( j \geq 1 \) a~little bit of algebra shows that
\begin{equation}
\label{eq:zawieranie}
(x,v) \in \Omega_j \quad \text{implies} \quad (1+xwv^{-1},w) \in \bigcup_{i=0}^{j-1} \Omega_i,
\end{equation}
where by convention \( \Omega_0 \coloneqq \Omega \).
Therefore, \( \tif \) is determined by induction: Having determined it on
\( \bigcup_{i=0}^j \Omega_i \), \( j \geq 0 \) for \( (x,v) \in \Omega_{j+1} \)
we determine \( \tif(x,v) \) by~\eqref{eq:ext}.
This completes the proof.
\end{proof}

The reasoning presented in the proof of Lemma~\ref{lem:uniq} suggests that given
\( f \in \lo \) we should study its extension \( \tif \) satisfying
\begin{equation}
\label{eq:ext-wzor}
\tif(x,v) = f(x,v)[x > 0] + \int_V \tif(1 + xwv^{-1},w) \mes(\dd w,v) [x \leq 0],
\end{equation}
for almost every \( (x,v) \in \tio \).
This extension is unique up to an equivalence class.
Here and subsequently we use the Iverson bracket notation~\cite[p.~24]{knuth},
that is, if \( P \) is a~statement that can be true or false, then
\[
[P] =
\begin{cases}
1, & \text{\( P \) is true},\\
0, & \text{otherwise}.
\end{cases}
\]

\begin{definition}
\label{def:ext}
We call the extension \( \tif \) satisfying~\eqref{eq:ext-wzor} the
\emph{boundary extension} of \( f \in \lo \).
\end{definition}

We stress that we do not assume that \( f \) belongs to \( D(A) \) in order to
define \( \tif \).
However, what is crucial, boundary extensions of functions from the domain of
\( A \) posses an important property which we describe in
Lemma~\ref{lem:dziedziny}.

Now we need to find a~suitable Banach space in which all extensions live.
For \( \omega \in \R \) we define the function \( \e_{\omega}\colon \R \to \R \)
by the formula \( \e_{\omega}(x) \coloneqq \e^{\omega x} \), \( x \in \R \).
Also, we introduce the product measure \( \lebm \times \nu \) on \( \tio \)
(recall~\eqref{eq:tio-def} for the definition of \( \tio \)).
We denote this product measure by \( \mu \), as on \( \Omega \).
Given \( \omega \geq 0 \) we let \( \ltio \) to be the space of equivalence
classes of \( \mu \)-measurable functions \( f \) on \( \tio \), such that
\[
\normtio{ f } \coloneqq \sup_{j \geq 0} \e^{-\omega j} \norm{ f
}_{L^1(\Gamma_j)} < +\infty,
\]
where
\begin{equation}
\label{eq:gammy}
\Gamma_j \coloneqq \bigcup_{i = 0}^j \Omega_i, \qquad j \geq 0
\end{equation}
with \( \Omega_0 \coloneqq \Omega \) and \( \Omega_i \), \( i \geq 1 \) defined
as in~\eqref{eq:omegi}.
Here, we naturally set \( \norm{ f }_{L^1(\Gamma_j)} \coloneqq \int_{\Gamma_j}
\abs{ f } \dmu \).
It is easy to check that \( \normtio{ \, \cdot \, } \) is a~norm on \( \ltio \),
and that \( \ltio \) equipped with this norm is a~Banach space.

\begin{lemma}
\label{lem:int-mes}
Let \( g \) be a~\( \nu \)-integrable function defined on \( V \).
Then
\begin{equation*}
\int_V \int_V vw^{-1} g(w) \mes(\dd w,v) \nuv = (p+q) \int_V g(w) \nuw.
\end{equation*}
\end{lemma}

\begin{proof}
The conclusion follows by~\eqref{eq:mes}, the Fubini theorem
and~\eqref{eq:rozklad-k}.
\end{proof}

\begin{lemma}
\label{lem:rozszerzenia-sa-porzadne}
Assume that
\begin{equation}
\label{eq:oszacowanie-omegi}
\omega > \max \lrp[\big]{ \log(p+q), 0 }.
\end{equation}
Then for \( f \in \lo \) its boundary extension \( \tif \) belongs to
\( \ltio \) and there exists \( M_{\omega} > 0 \) such that
\begin{equation}
\label{eq:norma-rozszerzenia}
\normtio{ \tif } \leq M_{\omega} \norm{ f }_{\lo}, \qquad f \in \lo.
\end{equation}
\end{lemma}

\begin{proof}
Let \( \omega \geq 0 \), \( f \in \lo \), and \( \tif \) be its boundary
extension.
For \( i \geq 1 \), \( v \in V \) we denote
\[
\Omega_{i,v} \coloneqq \lrc[\big]{ x \in \R\colon -ivb^{-1} < x \leq
  -(i-1)vb^{-1} }.
\]
Since \( x \in \Omega_{i,v} \) implies \( x \leq 0 \), it follows
by~\eqref{eq:ext-wzor} that
\begin{equation*}
\begin{split}
\int_{\Omega_i} \abs{ \tif } \dd \mu &= \int_V
\int_{\Omega_{i,v}} \abs{ \tif(x,v) } \dd x \nuv\\
&\leq \int_V \int_V \int_{\Omega_{i,v}} \abs{ \tif(1+xw v^{-1},w) } \dd x
\mes(\dd w,v) \nuv.
\end{split}
\end{equation*}
Changing variables \( x \mapsto 1 + xwv^{-1} \) leads to
\begin{equation*}
\begin{split}
\int_{\Omega_i} \abs{ \tif } \dd \mu &\leq \int_V \int_V
\int_{1+\Omega_{i,w}} vw^{-1} \abs{ \tif(x,w) } \dd x \mes(\dd w,v) \nuv\\
&= (p+q) \int_V \int_{1+\Omega_{i,v}} \abs{ \tif(x,w) } \dd x \nuw,
\end{split}
\end{equation*}
where \( 1+\Omega_{i,v} \) is the algebraic sum of \( \lrc{ 1 } \) and
\( \Omega_{i,v} \), with the last equality resulting from the Fubini theorem by
Lemma~\ref{lem:int-mes}.
Thus
\begin{equation}
\label{eq:oszacowanie-na-omega}
\begin{split}
\int_{\Omega_i} \abs{ \tif } \dmu \leq (p+q) \int_{1+\Omega_i} \abs{ \tif }
\dmu, \qquad i \geq 1,
\end{split}
\end{equation}
where \( 1 + \Omega_i \) is the algebraic sum of \( \set{1} \times V \) and
\( \Omega_i \).
Furthermore, we have
\begin{equation*}
1 + \Gamma_j \subset \Gamma_{j-1}, \qquad j \geq 1;
\end{equation*}
see Figure~\ref{fig:omegi-przesuniete} or use~\eqref{eq:zawieranie} with
  \( w \coloneqq v \).
\begin{figure}
\centering
\includegraphics[scale=1.8]{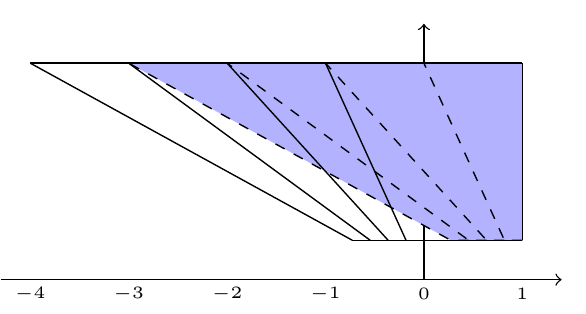}
\caption{The set \( \bigcup_{i=1}^4 (1+\Omega_i) \) is colored blue.}
\label{fig:omegi-przesuniete}
\end{figure} 
Combining this with~\eqref{eq:oszacowanie-na-omega}, for \( j \geq 1 \) we
obtain
\begin{equation*}
\int_{\Gamma_j} \abs{ \tif } \dmu = \normo{ f } +
\int_{\Gamma_j \setminus \Omega} \abs{ \tif } \dmu \leq \normo{ f } +
(p+q) \int_{\Gamma_{j-1}} \abs{ \tif } \dmu.
\end{equation*}
Hence, induction shows that
\begin{equation}
\label{eq:oszacowanie-normy-ext}
\int_{\Gamma_j} \abs{ \tif } \dmu \leq \normo{ f } \sum_{i=0}^j (p+q)^i, \qquad
j \geq 1.
\end{equation}
This implies that if \( p+q < 1 \), then for \( \omega \geq 0 \) we have
\begin{equation}
\label{eq:norma-ext-ogr}
\normtio{ \tif } \leq \frac{1}{1-p-q} \normo{ f }.
\end{equation}
On the other hand, if \( p+q > 1 \) and \( \omega \geq \log(p + q) \), then we
have \( \sup_{j \geq 0} \e^{-\omega j} (p+q)^j \leq 1 \) and
by~\eqref{eq:oszacowanie-normy-ext} it follows that
\[
\normtio{ \tif } \leq \sup_{j \geq 0} \e^{-\omega j} \frac{(p+q)^{j+1} -
  1}{p+q-1} \normo{ f } \leq \frac{p+q}{p+q-1} \normo{ f }.
\]
Finally, if \( p+q = 1 \) and \( \omega > 0 \), then again
by~\eqref{eq:oszacowanie-normy-ext}
\[
\normtio{ \tif } \leq \sup_{j \geq 0} (j+1)\e^{-\omega j} \normo{ f } \leq
\frac{1}{\omega} \e^{\omega-1} \normo{ f },
\]
which proves~\eqref{eq:norma-rozszerzenia}.
\end{proof}

We denote by \( \exts \) the set of all boundary extensions, that is,
\[
\exts \coloneqq \lrc{\tif\colon f \in \lo}.
\]
It is clear from~\eqref{eq:ext-wzor} that for \( f, g \in \lo \) and
\( \alpha \in \R \) we have \( \widetilde{\alpha f} = \alpha \tif \) and
\( \widetilde{f + g} = \tif + \tilde g \).
This implies that \( \exts \) is a~linear space and in view of
Lemma~\ref{lem:rozszerzenia-sa-porzadne} we have \( \exts \subseteq \ltio \)
provided that ~\eqref{eq:oszacowanie-omegi} holds.
Here and subsequently we fix such \( \omega \), that is,
\[
\omega > \max \lrp[\big]{ \log(p+q), 0 }.
\]
We endow \( \exts \) with \( \normtio{\, \cdot \,} \) norm, and define the
extension operator
\[
E\colon \lo \to \exts
\]
by
\[
Ef \coloneqq \tif.
\]

In order to set up notation, given Banach spaces \( X, Y \) we let
\( \lino{X,Y} \) to be the space of bounded linear operators \( X \to Y \) with
standard operator norm \( \norm{ \, \cdot \, }_{\lino{X,Y}} \).
If \( Y = X \) we write \( \lino{X,X} = \lino X \).

\begin{proposition}
\label{prop:operator-rozszerzania}
The operator \( E \) is an isomorphism between the spaces \( \lo \) and
\( \exts \).
Moreover
\begin{equation}
\label{eq:norma-operatora-ext}
\norm{ E }_{\lino{\lo, \ltio}} \leq M_{\omega},
\end{equation}
and
\begin{equation}
\label{eq:norma-operatora-ext-1}
\norm{ E^{-1} }_{\lino{\ltio, \lo}} = 1.
\end{equation}
where \( M_\omega \) is the constant from~\eqref{eq:norma-rozszerzenia}.
\end{proposition}

\begin{proof}
Inequality~\eqref{eq:norma-operatora-ext} follows directly from
Lemma~\ref{lem:rozszerzenia-sa-porzadne}.
On the other hand, since \( Ef = 0 \) implies \( 0 = \tif_{|\Omega} = f \), the
operator \( E \) is one-to-one.
Moreover, the inverse \( E^{-1} \) is the restriction operator, that is,
\( E^{-1} f = f_{|\Omega} \).
Hence
\[
\norm{ E^{-1} f }_{\lo} = \norm{ f_{|\Omega} }_{\lo} = \norm{ f
}_{L^1(\Gamma_0)} \leq \normt{ f }_{\ltio}, \qquad f \in \ltio.
\]
This shows~\eqref{eq:norma-operatora-ext-1} and completes the proof.
\end{proof}

Let now \( \tit = \tisem \) be the family of operators in \( \ltio \) given by
\begin{equation}
\label{eq:tisem}
\tit(t)f(x,v) \coloneqq f(x-tv,v), \qquad t \geq 0,\ (x,v) \in \tio,\ f \in
\ltio.
\end{equation}
A standard reasoning shows that \( \tit \) is a~strongly continuous semigroup
and its generator \( \tia \) is given by
\[
\tia f(x,v) \coloneqq -v \partial_x f(x,v), \qquad (x,v) \in \tio,\ f \in \ltio,
\]
with domain
\[
D(\tia) \coloneqq \wtio,
\]
where \( \wtio \) is the space of (equivalence classes of) functions
\( f \in \ltio \) satisfying:
\begin{enumerate}
  \item for almost every \( v \in V \) the function \( J \ni x \mapsto f(x,v) \)
is weakly differentiable,
  \item the function \( \Omega \ni (x,v) \mapsto v \partial_x f(x,v) \) belongs
to \( \ltio \).
\end{enumerate}
Analogously as in the case of \( \wo \), if \( f \in \wtio \), then for almost
every \( v \in V \) the function \( J \ni x \mapsto f(x,v) \) may be uniquely
extended to a~continuous function on \( (-\infty,1] \).

\begin{lemma}
\label{lem:niezmienniczosc}
The space \( \exts \) is invariant for the semigroup \( \tit \).
\end{lemma}

\begin{proof}
Fix \( t \geq 0 \).
Let \( f \in \lo \), and let \( \tif \in \exts \) be its boundary extension.
We prove that \( \tit(t)\tif \) is the boundary extension of \( g \in \lo \)
defined by
\[
g(x,v) = \tif(x-tv,v), \qquad (x,v) \in \Omega.
\]
We proceed by induction, showing \( \tilde g = \tit(t) \tif \) on each
\( \Gamma_j \), \( j \geq 0 \) (see~\eqref{eq:gammy}).
Let \( (x,v) \in \tio \) and recall that by~\eqref{eq:tisem},
\( \tit(t)\tif(x,v) = \tif(x-tv,v) \).
If \( (x,v) \in \Omega = \Gamma_0 \), then by~\eqref{eq:ext-wzor} we have
\[
\tilde g(x,v) = g(x,v) = \tit(t) \tif(x,v).
\]
Fix \( j \geq 0 \) and assume that \( \tilde g(x,v) = \tit(t) \tif(x,v) \) for
\( (x,v) \in \Gamma_j \).
If \( (x,v) \in \Omega_{j+1} = \Gamma_{j+1} \setminus \Gamma_j \), then for each
\( w \in V \) it follows that \( (1 + xwv^{-1},w) \in \Gamma_j \)
by~\eqref{eq:zawieranie}.
Therefore, by~\eqref{eq:ext-wzor},
\begin{equation*}
\begin{split}
\tilde g(x,v) &= \int_V \tilde g(1 + xwv^{-1},w) \mes(\dd w,v)\\
&= \int_V \tit(t)\tif(1 + xwv^{-1},w) \mes(\dd w,v)\\
&= \int_V \tif(1 + (x-tv)wv^{-1},w) \mes(\dd w,v)\\
&= \tif(x-tv,v)
\end{split}
\end{equation*}
for almost every \( (x,v) \in \Gamma_{j+1} \setminus \Gamma_j \).
This shows
\[
\tilde g(x,v) = \tit(t)\tif(x,v), \qquad (x,v) \in \Gamma_{j+1},
\]
which completes the proof.
\end{proof}

By Lemma~\ref{lem:niezmienniczosc} the part \( \tia_{\exts} \) of \( \tia \) in
\( \exts \), that is, the operator defined as
\[
\tia_{\exts} f \coloneqq \tia f, \qquad D(\tia_{\exts}) \coloneqq \lrc{ f \in
  D(\tia) \cap \exts\colon \tia f \in \exts },
\]
generates the strongly continuous semigroup \( \tisemy \) in
\( \exts \) given by
\begin{equation*}
\tit_{\exts}(t) \tif(x,v) \coloneqq \tif(x-tv,v), \qquad t \geq 0,\ (x,v) \in
\Omega,\ f \in \exts,
\end{equation*}
see for example~\cite[Corollary II.2.3]{engelnagel}.

\begin{lemma}
\label{lem:dziedziny}
Let \( f \in \lo \).
We have \( f \in D(A) \) if and only if \( \tif \in D(\tia_{\exts}) \).
\end{lemma}

\begin{proof}
Assume first that \( \tif \in D(\tia_{\exts}) \).
Of course \( f = \tif_{|\Omega} \in W^1(\Omega) \).
We need to show that~\eqref{eq:bc-mes} holds.
Let \( U \) be a~measurable subset of \( V \) such that
\( \nu(V \setminus U) = 0 \) and \( \tif_{|I \times U} \) is weakly
differentiable with respect to \( x \).
Without loss of generality we may assume that the function
\( [0,1] \ni x \mapsto f(x,v) \) is continuous for every \( v \in U \) and
\( f(x,v) = \tif(x,v) \) for \( (x,v) \in [0,1] \times U \).
Thus, for \( v \in U \) by~\eqref{eq:ext-wzor} we have
\begin{equation*}
f(0,v) = \tif(0,v) = \int_V \tif(1,w) \mes(\dd w,v) = \int_V f(1,w) \mes(\dd w,v),
\end{equation*}
proving that \( f \in D(A) \).

On the other hand, let \( f \in D(A) \).
The Hille-Yosida theorem implies that there exists \( \lambda_0 > 0 \) such that
for all \( \lambda > \lambda_0 \) the operator
\[
\lambda - \tia_{\exts}\colon D(\tia_{\exts}) \to \ltio
\]
is bijective.
Let \( \lambda > \lambda_0 \) and set
\[
F(x,v) := \lambda f(x,v) + v\partial_x f(x,v), \qquad (x,v) \in \Omega.
\]
Then \( F \in \lo \) since \( f \in D(A) \), and hence there exists
\( g \in D(\tia_{\exts}) \) satisfying
\begin{equation}
\label{eq:rown-resolwenty}
\lambda g - \tia_{\exts} g = \tilde F.
\end{equation}
By the first part of the proof it follows that \( E^{-1}g \in D(A) \).
Therefore, letting
\[
h \coloneqq E^{-1} g - f
\]
we see that \( h \in D(A) \) and \eqref{eq:rown-resolwenty} gives
\begin{equation}
\label{eq:slabe-mocne}
\lambda h(x,v) + v \partial_x h(x,v) = 0, \qquad x \in I
\end{equation}
for almost every \( v \in V \).
Thus, if we fix such \( v \in V \), then by~\cite[Corollary~3.1.6]{hormander}
the function \( I \ni x \mapsto h(x,v) \) is in fact a~classical solution
to~\eqref{eq:slabe-mocne}.
This implies
\[
h(x,v) = c_{\lambda,v} \e^{-\lambda x v^{-1}}, \qquad x \in I
\]
for a~constant \( c_{\lambda,v} \).
However, since~\eqref{eq:bc-mes} holds for \( h \), we get
\begin{equation*}
c_{\lambda,v} = \int_V c_{\lambda,w} \e^{-\lambda w^{-1}} \mes(\dd w,v).
\end{equation*}
Let
\[
C_{\lambda} \coloneqq \int_V v \abs{ c_{\lambda,v} } \nuv.
\]
We note that \( C_{\lambda} \) is finite.
Indeed, we have
\begin{equation*}
\begin{split}
\int_V \abs{ c_{\lambda,w} } \e^{-\lambda w^{-1}} \mes(\dd w,v) &= \int_V \int_I
\abs{ c_{\lambda,w} } \e^{-\lambda w^{-1} } \dd x \mes(\dd w,v)\\
&\leq \int_V \int_I \abs{ h(x,w) } \dd x \mes(\dd w,v),
\end{split}
\end{equation*}
the inequality resulting from
\( \e^{-\lambda w^{-1}} < \e^{-\lambda x w^{-1}} \) for \( x \in I \) and
\( w \in V \).
Then, by the Fubini theorem,
\begin{align*}
  C_{\lambda} &\leq \int_I \int_V \int_V v \abs{ h(x,w) } \mes(\dd w,v)
                \nuv \dd x\\
              &= (p+q) \int_I \int_V w \abs{ h(x,w) } \nuw \dd x\\
              &\leq (p+q) b \normo{ h },
\end{align*}
the equality being a~consequence of Lemma~\ref{lem:int-mes} with
\( g(w) = g_x(w) \coloneqq w \abs{ h(x,w) } \); this \( g \) is
\( \nu \)-integrable for almost every \( x \) because \( h \in \lo \).
Since \( \e^{-\lambda w^{-1}} < \e^{-\lambda b^{-1}} \) for \( w \in V \),
\[
C_{\lambda} \leq \e^{-\lambda b^{-1}} \int_V \int_V v \abs{ c_{\lambda,w} }
\mes(\dd w,v) \nuv =
(p+q) \e^{-\lambda b^{-1}} C_{\lambda}
\]
by Lemma~\ref{lem:int-mes} with \( g(w) \coloneqq w \abs{ c_{\lambda,w} } \);
this \( g \) is \( \nu \)-integrable because \( C_{\lambda} \) is finite.
This is true for all \( \lambda > \lambda_0 \) which leads to
\( C_{\lambda} = 0 \) for
\( \lambda > \lambda_0' \coloneqq \max\lrp[\big]{\lambda_0, b \ln(p+q)} \).
Hence \( c_{\lambda,v} = 0 \) for almost every \( v \in V \), provided
\( \lambda > \lambda_0' \).
Thus \( h = 0 \) and \( f = E^{-1}g \).
Finally, by the uniqueness of boundary extensions,
\[
\tif = EE^{-1}g = g \in D(\tia_{\exts}),
\]
which completes the proof.
\end{proof}

Now we are ready to prove the generation theorem.

\begin{proof}[Proof of Theorem~\textup{\ref{thm:glowne}}]
Let \( \sst = \sem \) be the family of linear operators in \( \lo \) defined by
\begin{equation}
\label{eq:kelvin}
T(t) \coloneqq E^{-1} \tit_{\exts} (t) E, \qquad t \geq 0.
\end{equation}
Then, see for example~\cite[7.4.22]{bobrowskiFA}, \( \sst \) is a strongly
continuous semigroup in \( \lo \) similar to \( \tisemy \).
Moreover, its generator is the operator \( E^{-1}\tia_{\exts} E \) with domain
\( E^{-1}D(\tia_{\exts}) \), which equals \( D(A) \) by
Lemma~\ref{lem:dziedziny}.
If \( f \in D(A) \), then by~\eqref{eq:ext-wzor} it follows that
\[
\partial_x \tif(x,v) = \partial_x f(x,v)
\]
for almost every \( (x,v) \in \Omega \).
Therefore
\[
(E^{-1}\tia_{\exts}E)f = Af, \qquad f \in D(A).
\]
This shows that \( A \) is the generator of \( \sst \) and the theorem follows.
\end{proof}

We call the semigroup generated by \( A \) the \emph{Rotenberg semigroup} and in
what follows denote it by \( \sst = \sem \).
By~\eqref{eq:kelvin} we have
\begin{equation}
\label{eq:wzor-na-polgrupe}
T(t)f(x,v) = \tif(x-tv,v), \qquad t \geq 0,\ (x,v) \in \Omega,\ f \in \lo,
\end{equation}
as conjectured in~\eqref{eq:sem}.
Furthermore, we introduce
\begin{equation}
\label{eq:xwv}
x_{wvt} \coloneqq 1 + x wv^{-1} - tw;
\end{equation}
for a~cell characterized by a~pair \( (x,v) \in \Omega \) at time
\( t \geq 0 \), \( x_{wvt} \) is the maturity parameter of a~potential mother of
the cell with maturation speed \( w \) at time \( 0 \), under proviso that
\( x_{wvt} \in I \).
Relations~\eqref{eq:ext-wzor} and~\eqref{eq:wzor-na-polgrupe} imply that given
\( t \geq 0 \) and \( f \in \lo \),
\begin{equation*}
T(t) f(x,v) = f(x-tv,v)[x > tv] + \int_V \tif(x_{wvt},w) \mes(\dd w,v) [x \leq tv]
\end{equation*}
for almost every \( (x,v) \in \Omega \).
In particular, let \( t \in [0,b^{-1}] \).
We have \( x_{wvt} > 1-tw > 1-wb^{-1} > 0 \) for \( (x,v) \in \Omega \) and
\( w \in V \).
Hence \( \tif(x_{wvt}) = f(x_{wvt}) \), and we may rewrite the above relation in
the form
\begin{equation}
\label{eq:T-small-t}
T(t) f(x,v) = f(x-tv,v)[x > tv] + \int_V f(x_{wvt},w) \mes(\dd w,v) [x \leq tv]
\end{equation}
for almost every \( (x,v) \in \Omega \), provided that \( t \in [0,b^{-1}] \).

It is worth noting that the extension \( \tif \) is nonnegative provided that
\( f \) is nonnegative, and hence \( T(t) \) is a~positive operator for
\( t \geq 0 \).
We use this fact later on.

\section{Growth estimates}
\label{sec:growth-estimates}

In this section we estimate the growth of the Rotenberg semigroup \( \sst \).
Let \( \sst^{*} = \set{ T^{*}(t)\colon t \geq 0 } \) be the dual semigroup of
\( \sst \), see~\cite[I.5.14]{engelnagel}.
First we find an explicit formula for the adjoint
\[
T^{*}(t)\colon \linf \to \linf
\]
of \( T(t) \) for \( t \in [0,b^{-1}] \), where \( \linf \) is the space of
(equivalence classes of) essentially bounded functions on \( \Omega \) with
standard essential supremum norm \( \norm{ \, \cdot \, }_{\linf} \).
We denote
\[
x_{wvt}^{*} \coloneqq (x + tv - 1) wv^{-1};
\]
for a~cell characterized by a~pair \( (x,v) \in \Omega \) at time \( 0 \),
\( x_{wvt}^{*} \) is the maturity parameter of a~potential daughter of the cell
with maturation speed \( w \) at time \( t \geq 0 \), under proviso that
\( x_{wvt}^{*} \in I \).
Note that \( x_{wvt}^{*} < 0 \) reflects the fact that the cell characterized by
\( (x,v) \) at time \( 0 \) will not mature fast enough to divide before time
\( t \).
Analogously, \( x_{wvt}^{*} \geq 0 \) means that the cell will divide before
time \( t \).

To simplify notation, let
\begin{equation*}
\tms(w,v) \coloneqq \tm(v,w), \qquad w,v \in V,
\end{equation*}
and for \( v \in V \) let \( \mess(v,\cdot) = \mess_{\nu}(v,\cdot) \) be the
measure on \( V \) defined by
\[
\mess(v,\dd w) \coloneqq p \tms(w,v) \nuw + q \delta_v(\dd w).
\]

\begin{lemma}
\label{lem:dual}
For \( t \in [0,b^{-1}] \) the adjoint operator \( T^{*}(t) \) of \( T(t) \) is
given by
\begin{equation}
\label{eq:wzor-na-dualny}
T^{*}(t) \phi(x,v) = \phi(x + tv,v) [x_{vvt}^{*} < 0] + \int_V
\phi(x_{wvt}^{*},w) \mess(v,\dd w) [0 \leq x_{vvt}^{*}]
\end{equation}
for almost every \( (x,v) \in \Omega \).
\end{lemma}

\begin{proof}
Let \( \phi \in \linf \) and fix \( t \in [0,b^{-1}] \).
By the definition of the adjoint operator
\begin{equation}
\label{eq:dualny}
\int_{\Omega} f T^{*}(t) \phi \dd \mu = \int_{\Omega} \phi T(t)f \dd \mu, \qquad
f \in \lo.
\end{equation}
By~\eqref{eq:T-small-t} and~\eqref{eq:mes}, we have
\begin{equation}
\label{eq:T-expl}
\int_{\Omega} \phi T(t) f \dd \mu = I_1 + I_2 + I_3,
\end{equation}
where
\begin{align*}
  I_1 &= \int_V \int_I \phi(x,v)f(x-tv)[x > tv] \dd x \nuv,\\
  I_2 &= p\int_V \int_I \int_V wv^{-1} \tm(w,v) \phi(x,v) f(x_{wvt},w) [x \leq tv]
        \nuw \dd x \nuv,\\
  I_3 &= q\int_V \int_I \phi(x,v) f(x_{vvt},v) [x \leq tv] \dd x \nuv.
\end{align*}
Changing variables \( x \mapsto x - tv \) we obtain
\[
I_1 = \int_V \int_I \phi(x+tv,v) f(x,v) [x_{vvt}^{*} < 0] \dd x \nuv,
\]
since \( [0 < x + tv < 1][x + tv > tv] = [0 < x < 1][x_{vvt}^{*} < 0] \) for
\( (x,v) \in \Omega \).
Similarly, changing variables \( x \mapsto x_{wvt} \), or equivalently
\( x_{vwt}^{*} \mapsto x \),
\begin{equation*}
I_2 = p \int_V \int_I \int_V \tm(w,v) \phi(x_{vwt}^{*},v) f(x,w) [0 \leq
x_{wwt}^{*}] \nuw \dd x \nuv,
\end{equation*}
and
\[
I_3 = q \int_V \int_I \phi(x_{vvt}^{*},v) f(x,v) [0 \leq x_{vvt}^{*}] \dd x
\nuv,
\]
since \( [0 < x_{vwt}^{*} < 1][x_{vwt}^{*} \leq tv] = [0 < x < 1][0 \leq
x_{wwt}^{*}] \) for \( (x,v) \in \Omega \) and \( w \in V \).
By~\eqref{eq:dualny}, \eqref{eq:T-expl}, and the definition of \( \tms \),
changing the order of integration in \( I_2 \), we
obtain~\eqref{eq:wzor-na-dualny}.
\end{proof}

We know that
\begin{equation}
\label{eq:rowne-normy}
\norm{ T(t) }_{\lino\lo} = \norm{ T^{*}(t) }_{\lino\linf}, \qquad t \geq 0.
\end{equation}
Since \( T(t) \) is a~positive operator for \( t \geq 0 \), the same is true for
\( T^{*}(t) \).
Hence, for \( \phi \in \linf \) such that \( \norm{ \phi }_{\linf} \leq 1 \) we
have
\[
-T^{*}(t) \ind_{\Omega} \leq -T^{*}(t) \abs{ \phi } \leq T^{*}(t) \phi \leq
T^{*}(t) \abs{ \phi } \leq T^{*}(t) \ind_{\Omega},
\]
where \( \ind_{\Omega} \) is the indicator function of \( \Omega \).
This leads to
\begin{equation*}
\begin{split}
\norm{ T^{*}(t) }_{\lino{\linf}} = \sup_{\substack{\phi \in \linf\\\norm{ \phi
  }_{\linf} \leq 1}} \norm{ T^{*}(t)\phi }_{\linf} = \norm{ T^{*}(t)
  \ind_{\Omega} }_{\linf}, \qquad t \geq 0,
\end{split}
\end{equation*}
and by~\eqref{eq:rowne-normy} we obtain
\begin{equation}
\label{eq:norma-inf}
\norm{ T(t) }_{\lino\lo} = \norm{ T^{*}(t) \ind_{\Omega} }_{\linf}, \qquad t
\geq 0.
\end{equation}
Finally, by~\eqref{eq:wzor-na-dualny} it follows that
\begin{equation}
\label{eq:dualny-na-1}
T^{*}(t) \ind_{\Omega}(x,v) = [x_{vvt}^{*} < 0] + (p+q) [0 \leq x_{vvt}^{*}]
\end{equation}
for \( t \in [0,b^{-1}] \) and almost every \( (x,v) \in \Omega \).

\begin{lemma}
\label{lem:norma-dla-malych-t}
We have
\[
\norm{ T(t) }_{\lino\lo} = \max(1,p+q), \qquad t \in (0,b^{-1}].
\]
\end{lemma}

\begin{proof}
Let \( t \in (0,b^{-1}] \).
Equality~\eqref{eq:dualny-na-1} shows that
\[
\norm{ T^{*}(t) \ind_{\Omega} }_{\linf} = \max(1,p+q)
\]
because the sets \( \set{(x,v) \in \Omega\colon x < 1 - tv} \) and
\( \set{(x,v) \in \Omega\colon 1-tv \leq x } \) are both
of positive \( \mu \)-measure.
By~\eqref{eq:norma-inf} this is the desired conclusion.
\end{proof}

Recall that a~positive operator \( S\colon \lo \to \lo \) is called
\emph{Markov} if \( \normo{ Sf } = 1 \) provided that \( f \in \lo \) is
nonnegative and \( \normo{ f } = 1 \).
In case \( p + q = 1 \) we may improve Lemma~\ref{lem:norma-dla-malych-t} as
follows.

\begin{lemma}
\label{lem:markowskosc}
Let \( t \in [0,b^{-1}] \) and assume that \( p + q = 1 \).
Then the operator \( T(t) \) is Markov.
\end{lemma}

\begin{proof}
Recall that the operator \( T(t) \) is positive.
Let \( f \in \lo \) be nonnegative and such that \( \normo{ f } = 1 \).
Then by~\eqref{eq:dualny},
\[
\normo{ T(t)f } = \int_{\Omega} \ind_{\Omega} T(t) f \dd \mu = \int_{\Omega} f
T^{*}(t) \ind_{\Omega} \dd \mu.
\]
However, by~\eqref{eq:dualny-na-1} we have
\( T^{*}(t) \ind_{\Omega} = \ind_{\Omega} \), thus
\( \normo{ T(t)f } = \normo{ f } \), which completes the proof.
\end{proof}

\begin{theorem}
\label{thm:norma-polgrupy}
Let \( t \geq 0 \).
\begin{enumerate}
  \item\label{item:6} If \( p+q > 1 \), then
\begin{equation}
\label{eq:upper-bound}
\norm{ T(t) }_{\lino\lo} \leq (p+q)^{\ceil{tb}},
\end{equation}
where \( \ceil{tb} \) is the smallest integer larger than or equal to \( tb \).
  \item\label{item:5} If \( p+q = 1 \), then the operator \( T(t) \) is Markov.
  \item\label{item:grube} If \( p + q < 1 \), then
\[
\norm{ T(t) }_{\lino\lo} \leq 1.
\]
\end{enumerate}
\end{theorem}

\begin{proof}
Let \( n \coloneqq \ceil{tb} \) and \( s \coloneqq t/n \).
Hence \( s \in (0,b^{-1}] \).
By the semigroup property and Lemma~\ref{lem:norma-dla-malych-t} we have
\begin{equation*}
\norm{ T(t) }_{\lino\lo} \leq \norm{ T(s) }_{\lino\lo}^n \leq \max(1,p+q)^n,
\end{equation*}
which proves~\ref{item:6} and~\ref{item:grube}.

In order to show~\ref{item:5} we fix \( n \) and \( s \) as above.
Then by Lemma~\ref{lem:markowskosc} the operator \( T(s) \) is Markov, and so is
\( T(t) = [T(s)]^n \) as a~power of a~Markov operator.
\end{proof}

\begin{theorem}
\label{thm:norma-polgrupy-leb}
Assume that \( p + q < 1 \) and \( V = (a,b) \) with the Lebesgue measure.
\begin{enumerate}
  \item\label{item:3} If \( a = 0 \), then
\[
\norm{ T(t) }_{\lino\lo} = 1, \qquad t \geq 0.
\]
  \item\label{item:4} If \( a > 0 \), then
\begin{equation}
\label{eq:tlea}
\norm{ T(t) }_{\lino\lo} = 1, \qquad 0 \leq t < a^{-1},
\end{equation}
and
\begin{equation}
\label{eq:oszacowanie-t}
\norm{ T(t) }_{\lino\lo} \leq (p+q)^{\floor{ta}}, \qquad t \geq a^{-1},
\end{equation}
where \( \floor{ta} \) is the largest integer less than or equal to \( ta \). 
\end{enumerate}
\end{theorem}

Before we prove Theorem~\ref{thm:norma-polgrupy-leb} we state a~couple of
auxiliary results.

\begin{lemma}
\label{lem:od-dolu}
Under the hypotheses of Theorem~\textup{\ref{thm:norma-polgrupy-leb}} we have
\begin{equation}
\label{eq:norma-T-od-dolu}
\norm{ T^{*}(t) }_{\lino\linf} \geq \esssup_{(x,v) \in \Omega} [x_{vvt}^{*} <
0], \qquad t > 0.
\end{equation}
\end{lemma}

\begin{proof}
For each \( t_1 \in [0,b^{-1}] \) by~\eqref{eq:dualny-na-1} we have
\[
T^{*}(t_1) \ind_{\Omega}(x,v) = [x_{vvt_1}^{*} < 0] + (p+q) [0 \leq
x_{vvt_1}^{*}] \geq [x_{vvt_1}^{*} < 0]
\]
for almost every \( (x,v) \in \Omega \).
Therefore, given \( t_2 \in [0,b^{-1}] \), by the semigroup property for
\( T^{*} \), the positivity of \( T^{*}(t_2) \), and~\eqref{eq:wzor-na-dualny},
\begin{equation} 
\label{eq:indukcja-od-dolu}
\begin{split}
T^{*}(t_1 + t_2) \ind_{\Omega} (x,v) &\geq T^{*}(t_2)[x_{vvt_1}^{*} < 0]\\
&\geq\lrb[\big]{ (x+t_2v)_{vvt_1}^{*} < 0}[x_{vvt_2}^{*} < 0]\\
&= [x_{vv(t_1+t_2)}^{*} < 0]
\end{split}
\end{equation}
for almost every \( (x,v) \in \Omega \).

For \( t > 0 \) choose \( N \geq 0 \), \( s \in (0,b^{-1}] \) such that
\begin{equation*}
t = Nb^{-1} + s.
\end{equation*}
Inequality~\eqref{eq:indukcja-od-dolu} implies
\[
T^{*}(t) \ind_{\Omega}(x,v) = T(Nb^{-1}+s) \ind_{\Omega}(x,v) \geq [x_{vvt}^{*}
< 0]
\]
for almost every \( (x,v) \in \Omega \) by induction.
Hence~\eqref{eq:norma-T-od-dolu} follows by~\eqref{eq:norma-inf}.
\end{proof}

\begin{lemma}
\label{lem:oszacowanie-t}
Under the hypotheses of Theorem~\textup{\ref{thm:norma-polgrupy-leb}}, if
\( a > 0 \), then
\begin{equation}
\label{eq:t-estimate}
T^{*}(t) \ind_{\Omega}(x,v) \leq (p+q)^j, \qquad t > 0
\end{equation}
for almost every \( (x,v) \in \Omega \), where
\begin{equation}
\label{eq:j}
j = j(x,v,t) \coloneqq \min \lrc{ i \geq 0\colon x_{vv(t-ia^{-1})}^{*} < 0}.
\end{equation}
\end{lemma}

Note that for \( (x,v) \in \Omega \), in view of the interpretation of
\( x_{vv(t-ia^{-1})}^{*} \) from the beginning of this section, \( j \) is the
least nonnegative integer such that a~cell characterized by a~pair \( (x,v) \)
at time \( 0 \) will not divide before time \( \max(t - ja^{-1},0) \).

\begin{proof}
For \( i \geq 0 \) and \( t > 0 \) define \( \phi_{i,t} \colon \Omega \to \R \)
by
\[
\phi_{i,t}(x,v) \coloneqq \lrb{ (i-1)va^{-1} \leq x_{vvt}^{*} < iva^{-1} },
\qquad (x,v) \in \Omega,
\]
and \( \psi_{i,t,w}\colon \Omega \to \R \) for \( w \in V \) by
\[
\psi_{i,t,w}(x,v) \coloneqq \lrb{ (i-1)va^{-1} + vw^{-1} \leq x_{vvt}^{*} <
  iva^{-1} + vw^{-1} }, \qquad (x,v) \in \Omega.
\]
Also, denote
\[
r \coloneqq p + q.
\]

\emph{Step} 1.
Inequality~\eqref{eq:t-estimate} is equivalent to
\begin{equation}
\label{eq:indukcja-zalozenie}
T^{*} \lrp{ t } \ind_{\Omega} \leq \sum_{i=0}^{+\infty} r^i
\phi_{i,t}, \qquad t > 0.
\end{equation}
Indeed, let \( t > 0 \).
The functions \( \phi_{i,t} \) for \( i \geq 0 \) are indicator functions of
disjoint sets whose union equals \( \Omega \), since
\begin{equation}
\label{eq:i0}
x_{vvt}^{*} = x + tv - 1 > -1 > -va^{-1}, \qquad (x,v) \in \Omega.
\end{equation}
Hence, for fixed \( (x,v) \in \Omega \), exactly one term in the series is
nonzero at \( (x,v) \), that is, there exists a~unique \( m \geq 0 \) such that
\( \phi_{m,t}(x,v) = 1 \).
For \( i \geq 0 \) we have
\begin{equation}
\label{eq:phi-inna-def}
\phi_{i,t}(x,v) = [-va^{-1} \leq x_{vvt}^{*} - iva^{-1} < 0] = [-va^{-1} \leq
x_{vv(t-ia^{-1})}^{*} < 0],
\end{equation}
thus taking \( i = m \) we see that \( j \leq m \), where \( j \) is defined
by~\eqref{eq:j}.
In the case \( m = 0 \), clearly \( j = m \).
On the other hand, if \( m \geq 1 \), then \( \phi_{i,t}(x,v) = 0 \) for
\( 0 \leq i < m \), which by~\eqref{eq:phi-inna-def} implies \( j \geq m \).
Hence \( j = m \), and finally
\[
\sum_{i=0}^{+\infty} r^i \phi_{i,t}(x,v) = r^m = r^j
\]
as desired.

\emph{Step} 2.
Let \( t > 0 \) and \( w \in V \).
We estimate the sum \( \sum_{i=0}^{+\infty} r^i \psi_{i,t,w}(x,v) \) under
proviso \( x_{vvt}^{*} \geq 0 \).
To this end we fix \( (x,v) \in \Omega \) such that \( x_{vvt}^{*} \geq 0 \), and
for \( i \geq 0 \) we define
\[
\Psi_{i,t,w}(x,v) \coloneqq [iva^{-1} \leq x_{vvt}^{*} < iva^{-1} + vw^{-1}].
\]
Assuming that real numbers \( \alpha \), \( \beta \), \( \gamma \) and
\( \delta \) satisfy
\begin{equation}
\label{eq:alpha-beta}
\alpha \leq \beta \leq \gamma \leq \delta,
\end{equation}
we have
\begin{equation}
\label{eq:iv-ineq}
[\alpha \leq y < \beta] + r [\beta \leq y < \delta] \leq [\alpha \leq y <
\gamma] + r [\gamma \leq y < \delta], \qquad y \in \R,
\end{equation}
since \( r < 1 \).
Taking
\[
\alpha \coloneqq iva^{-1}, \quad \beta \coloneqq iva^{-1} + vw^{-1}, \quad
\gamma \coloneqq iva^{-1} + va^{-1}, \quad \delta \coloneqq (i+1)va^{-1} +
vw^{-1},
\]
condition \eqref{eq:alpha-beta} holds for \( i \geq 0 \) because
\( 0 < w^{-1} < a^{-1} \).
Hence, by~\eqref{eq:iv-ineq} with \( y \coloneqq x_{vvt}^{*} \) we obtain
\begin{equation*}
r^i \Psi_{i,t,w}(x,v) + r^{i+1} \psi_{i+1,t,w}(x,v) \leq r^i \phi_{i+1,t}(x,v) +
r^{i+1} \Psi_{i+1,t,w}(x,v), \qquad i \geq 0.
\end{equation*}
Summing this inequality for \( i \geq 0 \) we get
\begin{equation*}
\sum_{i=0}^{+\infty} r^i \Psi_{i,t,w}(x,v) + \sum_{i = 1}^{+\infty} r^i
\psi_{i,t,w}(x,v) \leq \sum_{i = 0}^{+\infty} r^i \phi_{i+1,t}(x,v) + \sum_{i =
  1}^{+\infty} r^i \Psi_{i,t,w}(x,v);
\end{equation*}
note that here all sums are finite since \( \phi_{i,t} \)'s,
\( \psi_{i,t,w} \)'s and \( \Psi_{i,t,w} \)'s are indicator functions of
disjoint sets.
Thus
\begin{equation*}
\Psi_{0,t,w}(x,v) + \sum_{i=1}^{+\infty} r^i
\psi_{i,t,w}(x,v) \leq \sum_{i = 0}^{+\infty} r^i \phi_{i+1,t}(x,v).
\end{equation*}
Recall that \( x_{vvt}^{*} \geq 0 \), therefore
\begin{equation*}
\psi_{0,t,w}(x,v) = [-va^{-1} + vw^{-1} \leq x_{vvt}^{*} < 0] + [0 \leq
x_{vvt}^{*} < vw^{-1}] = \Psi_{0,t,w}(x,v).
\end{equation*}
Hence finally
\begin{equation}
\label{eq:psi-sum}
\sum_{i=0}^{+\infty} r^i \psi_{i,t,w}(x,v) \leq \sum_{i=0}^{+\infty}
r^i \phi_{i+1,t}(x,v)
\end{equation}
provided that \( t > 0 \), \( w \in V \), and \( (x,v) \in \Omega \) satisfies
\( x_{vvt}^{*} \geq 0 \).

\emph{Step} 3.
Let \( s > 0 \) and set
\[
t \coloneqq b^{-1} + s.
\]
We find a~formula for \( T^{*}(b^{-1}) \phi_{i,s} \).
For \( (x,v) \in \Omega \) we have \( (x+vb^{-1})_{vvs}^{*} = x_{vvt}^{*} \) and
\[
(x_{wvb^{-1}}^{*})_{wws}^{*} = (x+vb^{-1}-1)wv^{-1} + sw - 1 = x_{vvt}^{*}
wv^{-1} - 1.
\]
Therefore,
\begin{equation*}
\phi_{i,s}(x+vb^{-1},v) = \phi_{i,t}(x,v), \qquad i \geq 0,
\end{equation*}
and
\begin{equation*}
\phi_{i,s}(x_{wvb^{-1}}^{*},w) = \psi_{i,t,w}(x,v), \qquad i \geq 0,\ w \in V.
\end{equation*}
Combining these relations with~\eqref{eq:wzor-na-dualny}, we have
\begin{equation}
\label{eq:t-na-phi0}
T^{*}(b^{-1}) \phi_{i,s}(x,v) = \phi_{i,t}(x,v) [x_{vvb^{-1}}^{*} < 0]
+ \int_V \psi_{i,t,w}(x,v) \mess(v,\dd w) [0 \leq x_{vvb^{-1}}^{*}]
\end{equation}
for \( i \geq 0 \) and almost every \( (x,v) \in \Omega \).

\emph{Step} 4.
We show that~\eqref{eq:indukcja-zalozenie} holds for every
\( t \in \bigl( (n-1)b^{-1}, nb^{-1} \bigr] \), \( n \geq 1 \) by induction on
\( n \).
For \( t \in (0,b^{-1}] \) we have
\[
x_{vvt}^{*} = x + tv - 1 < 1 + vb^{-1} - 1 = vb^{-1} < va^{-1}, \qquad (x,v) \in
\Omega,
\]
thus \( [0 \leq x_{vvt}^{*}] = [0 \leq x_{vvt}^{*} < va^{-1}] = \phi_{1,t}(x,v)
\).
Similarly, by~\eqref{eq:i0},
\[
[x_{vvt}^{*} < 0] = \phi_{0,t}(x,v), \qquad (x,v) \in \Omega.
\]
Hence, by~\eqref{eq:dualny-na-1},
\begin{equation*}
T^{*} \lrp{ t } \ind_{\Omega} = \phi_{0,t} + r \phi_{1,t} = \sum_{i=0}^{+\infty}
r^i \phi_{i,t}.
\end{equation*}
This shows that~\eqref{eq:indukcja-zalozenie} holds for \( t \in (0,b^{-1}] \).

In order to perform the inductive step let \( n \geq 1 \) and assume
that~\eqref{eq:indukcja-zalozenie} holds for each
\( t \in \bigl((n-1)b^{-1},nb^{-1}\bigr] \).
Fix \( t \in \bigl(nb^{-1},(n+1)b^{-1}\bigr] \), and choose
\( s \in \bigl((n-1)b^{-1},nb^{-1}\bigr] \) such that
\[
t = b^{-1} + s.
\]
Since \( r < 1 \) we have \( \sum_{i=0}^{+\infty} r^i \phi_{i,t}(x,v) \leq 1 \)
for \( (x,v) \in \Omega \) (recall that exactly one term of the series is
nonzero).
Therefore \( \sum_{i=0}^{+\infty} r^i \phi_{i,t} \) is an element of
\( \linf \).
Hence, by~\eqref{eq:indukcja-zalozenie} with \( t \) replaced by \( s \), and by
the fact that \( T^{*}(b^{-1}) \) is a~positive and bounded operator,
\begin{equation*}
T^{*}(t) \ind_{\Omega} = T^{*}(b^{-1}) T^{*}(s) \ind_{\Omega} \leq
\sum_{i=0}^{+\infty} r^i T^{*}(b^{-1}) \phi_{i,s}.
\end{equation*}
Combining this with~\eqref{eq:t-na-phi0}, we obtain
\begin{equation}
\label{eq:ts-leq}
\begin{split}
T^{*}(t) \ind_{\Omega}(x,v) &\leq [x_{vvb^{-1}}^{*} < 0] \sum_{i=0}^{+\infty}
r^i \phi_{i,t}(x,v)\\
&\pe + [0 \leq x_{vvb^{-1}}^{*}] \int_V \sum_{i=0}^{+\infty} r^i
\psi_{i,t,w}(x,v) \mess(v,\dd w)
\end{split}
\end{equation}
for almost every \( (x,v) \in \Omega \).

Denote
\[
\chi(x,v) \coloneqq [x_{vvb^{-1}}^{*} < 0], \qquad (x,v) \in \Omega.
\]
For \( (x,v) \in \Omega \), if \( x_{vvb^{-1}}^{*} \geq 0 \), or equivalently
\( 1-\chi(x,v) = 1 \), then \( x_{vvt}^{*} = x_{vvb^{-1}}^{*} + sv \geq 0 \).
Hence, by~\eqref{eq:ts-leq} and \eqref{eq:psi-sum}, we obtain
\begin{equation*}
\begin{split}
T^{*}(t) \ind_{\Omega} &\leq \chi \sum_{i=0}^{+\infty} r^i \phi_{i,t} + (1-\chi)
\int_V \mess(v,\dd w) \sum_{i=0}^{+\infty} r^i
\phi_{i+1,t}\\
&= \chi \sum_{i=0}^{+\infty} r^i \phi_{i,t} + (1-\chi) \sum_{i=0}^{+\infty}
r^{i+1} \phi_{i+1,t}\\
&= \chi \phi_{0,t} + \sum_{i=1}^{+\infty} r^i \phi_{i,t},
\end{split}
\end{equation*}
where in the first equality we used \( \int_V \mess(v,\dd w) = r \) for
\( v \in V \).
However, we see that \( \chi \phi_{0,t} = \phi_{0,t} \), since
\( x_{vvb^{-1}}^{*} \leq x_{vvt}^{*} \) for each \( (x,v) \in \Omega \).
Therefore~\eqref{eq:indukcja-zalozenie} follows, and the proof is complete.
\end{proof}

\begin{proof}[Proof of Theorem~\textup{\ref{thm:norma-polgrupy-leb}}]
For part~\ref{item:3} and~\eqref{eq:tlea} we argue as follows.
Fix \( t > 0 \) and suppose that \( a = 0 \) or \( 0 < a < t^{-1} \).
Then the set \( \lrc{ (x,v) \in \Omega\colon x_{vvt}^{*} < 0 } \) is of positive
Lebesgue measure.
Indeed, the set \( U \coloneqq V \cap (a,t^{-1}) \) is an open interval, and for
\( v \in U \) the Lebesgue measure of the set of \( x \in I \) satisfying
\( x + tv < 1 \) is positive, being equal \( 1 - tv > 0 \).
Combining this with Lemma~\ref{lem:od-dolu}, we have
\( \norm{ T^{*}(t) }_{\lino\linf} \geq 1 \).
Hence, by~\eqref{eq:rowne-normy},
\[
\norm{ T(t) }_{\lino\lo} \geq 1.
\]
However, we know from Theorem~\ref{thm:norma-polgrupy}~\ref{item:grube} that
\( \norm{ T(t) }_{\lino\lo} \leq 1 \).
This means that
\[
\norm{ T(t) }_{\lino\lo} = 1
\]
provided that \( a = 0 \) or \( 0 < a < t^{-1} \), which proves~\ref{item:3}
and~\eqref{eq:tlea}.

In order to prove~\eqref{eq:oszacowanie-t}, for \( t \geq a^{-1} \) let \( m \)
be the unique positive integer satisfying
\[
ma^{-1} \leq t < (m+1)a^{-1},
\]
and let \( (x,v) \in \Omega \).
Then
\[
x_{vvt}^{*} = x + tv - 1 \geq x + mva^{-1} - 1 > mva^{-1} - 1 > (m-1)va^{-1},
\]
since \( va^{-1} > 1 \).
This implies that
\[
x_{vv(t-ia^{-1})}^{*} = x_{vvt}^{*} - iva^{-1} > 0, \qquad 0 \leq i \leq m-1,
\]
hence \( j(x,v,t) = \min \lrc{ i \geq 0\colon x_{vv(t-ia^{-1})}^{*} < 0 } \geq m
\).
Thus, since \( p + q < 1 \), Lemma~\ref{lem:oszacowanie-t} implies
\[
\norm{T^{*}(t) \ind_{\Omega}}_{\linf} \leq (p+q)^m = (p+q)^{\floor{ta}}
\]
and the proof is complete by~\eqref{eq:norma-inf}.
\end{proof}

A~natural question arises whether the estimates in
Theorem~\ref{thm:norma-polgrupy}~\ref{item:6} and
Theorem~\ref{thm:norma-polgrupy-leb}~\ref{item:4} are optimal.
As we shall see, the answer is generally in negative; the growth (resp.
decay) is in fact slower (resp.
faster) than Theorem~\ref{thm:norma-polgrupy}~\ref{item:6} and
Theorem~\ref{thm:norma-polgrupy-leb}~\ref{item:4} may suggest.
Unfortunately, the exact growth and decay rates seem to depend in crucial way on
an interplay of parameters \( a, b \) and kernel \( k \), and thus an explicit
formula, if it exists, evades us.
We can show, however, that equality in~\eqref{eq:upper-bound} holds rather
seldom (a similar argument applies to~\eqref{eq:oszacowanie-t}).
For the sake of this argument, we restrict ourselves to the case \( V = (a,b) \)
with the Lebesgue measure and assume that
\[
r \coloneqq p + q > 1.
\]

We begin by finding necessary and sufficient conditions for
\begin{equation}
\label{eq:Tr2}
\norm{ T(2b^{-1}) }_{\lino\lo} = r^2
\end{equation}
to hold.
To simplify notation we let \( s \coloneqq b^{-1} \) and \( t \coloneqq 2s \).
Fix \( (x,v) \in \Omega \), \( w \in V \).
Note that \( x_{vvs}^{*} = x_{vvt}^{*} - sv \),
\( (x+sv)_{vvs}^{*} = x_{vvt}^{*} \) and
\[
(x_{wvs}^{*})_{wws}^{*} = (x+sv-1)wv^{-1} + sw - 1 = (x+2sv-1)wv^{-1} - 1 =
x_{vvt}^{*}wv^{-1} - 1.
\]
Therefore 
\begin{align*}
  \lrb[\big]{(x+sv)_{vvs}^{*} < 0} [x_{vvs}^{*} < 0]
  &= [x_{vvt}^{*} < 0],\\
  \lrb[\big]{0 \leq (x+sv)_{vvs}^{*}} [x_{vvs}^{*} < 0]
  &= [0 \leq x_{vvt}^{*} < sv],\\
  \lrb[\big]{(x_{wvs}^{*})_{wws}^{*} < 0} [0 \leq x_{vvs}^{*}]
  &= [sv \leq x_{vvt}^{*} < vw^{-1}],\\
  \lrb[\big]{0 \leq (x_{wvs}^{*})_{wws}^{*}} [0 \leq x_{vvs}^{*}]
  &= [vw^{-1} \leq x_{vvt}^{*}].
\end{align*}
Hence, using the semigroup property for \( T^{*}(t) = T^{*}(s + s) \),
by~\eqref{eq:wzor-na-dualny} and \eqref{eq:dualny-na-1} we have
\begin{equation}
\label{eq:t2}
\begin{split}
T^{*}(t) \ind_{\Omega}(x,v)
&= [x_{vvt}^{*} < 0] + r [0 \leq x_{vvt}^{*} < sv]\\
&\pe + p \int_V \tms(w,v) [sv \leq x_{vvt}^{*} < vw^{-1}] \dd w + q [sv \leq
x_{vvt}^{*} < 1]\\
&\pe + pr \int_V \tms(w,v) [vw^{-1} \leq x_{vvt}^{*} ] \dd w + qr [1 \leq
x_{vvt}^{*}]
\end{split}
\end{equation}
for almost every \( (x,v) \in \Omega \).

Let \( v \in V \) and denote
\[
I_v \coloneqq \set{x \in I\colon x_{vvt}^{*} \geq sv} = \set{x \in I\colon x
  \geq 1 - sv}.
\]
If \( x \in I_v \) and \( y \in I \setminus I_v \), then by~\eqref{eq:t2} we
have
\[
T^{*}(t) \ind_{\Omega}(y,v) \leq r \leq T^{*}(t) \ind_{\Omega}(x,v).
\]
Hence, since the Lebesgue measure of \( I_v \) is positive, being equal
\( sv \),
\begin{equation}
\label{eq:Tstar-norm}
\norm{ T^{*}(t) }_{\lino\linf} = \esssup_{v \in V,\ x \in I_v} T^{*}(t)
\ind_{\Omega}(x,v).
\end{equation}
Define
\[
c(x,v) \coloneqq \max \lrp[\Big]{ a, \frac{v}{x_{vvt}^{*}} }, \qquad v \in V,\ x
\in I_v;
\]
note that \( a \leq c(x,v) \leq b \).
Then, using~\eqref{eq:t2}, for almost every \( (x,v) \in \Omega \) such that
\( x \in I_v \) we obtain
\begin{equation}
\label{eq:TCC}
T^{*}(t) \ind_{\Omega}(x,v) = C(x,v) + D(x,v),
\end{equation}
where
\[
C(x,v) \coloneqq p \int_a^{c(x,v)} \tms(w,v) \dd w + pr \int_{c(x,v)}^b \tms(w,v)
\dd w
\]
and
\[
D(x,v) \coloneqq q [sv \leq x_{vvt}^{*} < 1] + qr [1 \leq x_{vvt}^{*}].
\]
Let
\[
d \coloneqq \max\lrp{a,t^{-1}} = \max(a, b/2),
\]
and observe that \( a \leq d < b \).
For \( v \in V \) the set
\[
\set{x \in I\colon x_{vvt}^{*} \geq 1} = \set{x \in I\colon x \geq 2 - tv}
\]
is of positive Lebesgue measure if and only if \( tv > 1 \), or equivalently
\( v \in (d,b) \).
Consequently, if \( p = 0 \) then~\eqref{eq:Tr2} holds by~\eqref{eq:Tstar-norm}
and~\eqref{eq:TCC}.
Suppose now that \( p > 0 \) and let \( v \in V \).
As a~function of \( x \in I_v = [1-sv,1) \), \( x \mapsto x_{vvt}^{*} \) is
increasing and converges to \( tv \) as \( x \to 1^{-} \).
This implies that \( I_v \ni x \mapsto c(x,v) \) is decreasing and converges to
\( d \) as \( x \to 1^{-} \).
Hence
\[
\esssup_{v \in V\, x \in I_v} C(x,v) = pr
\]
if and only if the following condition holds:
\begin{enumeratev}
  \item[\namedlabel{i:ve}{(V\(_{\epsilon}\))}] For each \( \epsilon > 0 \) there
exists a~Lebesgue measurable subset \( U \) of \( V \) with a~positive measure
such that
\begin{equation*}
\int_d^b \tms(w,v) \dd w > 1-\epsilon, \qquad v \in U.
\end{equation*}
\end{enumeratev}
Therefore, again by~\eqref{eq:Tstar-norm} and~\eqref{eq:TCC}, if \( q = 0 \),
then~\eqref{eq:Tr2} is equivalent to~\ref{i:ve}.
Finally, if \( p,q > 0 \), then~\eqref{eq:Tr2} holds if and only if~\ref{i:ve}
holds with the additional requirement \( U \subseteq (d,b) \).

Note that by the definition of \( \tms \) and~\eqref{eq:rozklad-k}
condition~\ref{i:ve} is trivially satisfied (for any \( U \subseteq V \) with
positive Lebesgue measure) provided that \( d = a \), which is equivalent to
\( b/2 \leq a \).

Continuing the procedure described above for \( t \coloneqq nb^{-1} \),
\( n \geq 2 \) we can generalize this result.
To this end, we fix \( n \geq 2 \) and introduce condition:
\begin{enumeratev}
  \item[\namedlabel{i:ven}{(V\(_{\epsilon}^n \))}] For each \( \epsilon > 0 \)
there exists a~Lebesgue measurable subset \( U \) of \( V \) with a~positive
measure such that
\begin{equation*}
\int_{d_n}^b \tms(w,v) \dd w > 1-\epsilon, \qquad v \in U,
\end{equation*}
where
\[
d_n \coloneqq \max \lrp[\Big]{ a, \frac{nb}{n+1} }, \qquad n \geq 2.
\]
\end{enumeratev}
Then we can check that that the following result holds.

\begin{proposition}
\label{prop:kiedy-max}
Assume that \( p + q > 1 \), and \( V = (a,b) \) with the Lebesgue measure.
For \( n \geq 2 \) we have
\[
\norm{ T(nb^{-1}) }_{\lino\lo} = (p+q)^n
\]
if and only if
\begin{enumerate}
  \item \( p = 0 \), or
  \item \( q = 0 \) and condition~\textup{\ref{i:ven}} holds, or
  \item\label{i:venplus} condition~\textup{\ref{i:ven}} holds with additional requirement
\( U \subseteq (d_n,b) \).
\end{enumerate}
\end{proposition}

Note that if \( p = 0 \), then all cells degenerate and the structure of the
population does not change in time, and so the model is quite uninteresting.
On the other hand, condition~\ref{i:ven} is very restrictive.
Thus the proposition shows that the case when equality in Theorem
\eqref{thm:norma-polgrupy}~\ref{item:6} holds is rather rare.
For example, assuming \( p,q > 0 \), condition~\ref{i:venplus} is satisfied for
any \( t = nb^{-1} \), \( n \geq 2 \), provided that
\[
\tms(w,v) \coloneqq \frac{1}{b-v} [v < w < b], \qquad v,w \in U \cap V,
\]
where \( U \subset \R^2 \) is a~neighbourhood of \( (b,b) \).
For such \( \tms \), \( \tm \) does not even satisfy assumptions of Boulanouar's
theorem~\cite[Theorem~6.1]{boulanouar}, and seems to be rather uninteresting
biologically.
For this would mean that if a~mother cell matures sufficiently fast, then all
its daughter cells would need to mature even faster.

\section{Asymptotic behavior}
\label{sec:asymptotic-stability}

This section is devoted to the asymptotic behaviour of the Rotenberg semigroup.
Throughout this section we assume that \( V = (a,b) \), and \( \nu \) is the
Lebesgue measure.
At first we state Boulanouar's result.

Let \( \omega_0(\sst) \) be the \emph{growth bound} (or \emph{type}) of the
semigroup \( \sst \) defined as
\begin{equation*}
\inf \lrc{ \omega \in \R\colon \sup_{t \geq 0} \norm{ \e^{-\omega t} T(t)
  }_{\lino\lo} < +\infty },
\end{equation*}
or equivalently
\begin{equation*}
\lim_{t \to +\infty} t^{-1} \log \norm{ T(t) }_{\lino\lo}.
\end{equation*}

\begin{theorem}[{\cite[Theorem~6.1]{boulanouar}}]
\label{thm:boulanouara}
Assume that \( 0 < 1-q < p \) and the following conditions hold:
\begin{enumerateh}[series=h]
  \item\label{item:irr} For every measurable set\/ \( U \subset V \) such
that\/ \( \leb(U) > 0 \) and\/ \( \leb(V \setminus U) > 0 \) we have
\begin{equation*}
\int_{V \setminus U} \int_{U} \tm(w,v) \dd v \dd w > 0.
\end{equation*}
  \item\label{item:bounded} The kernel \( \tm \) is essentially bounded on
\( V \times V \).
\end{enumerateh}
Then there exists a~rank one projection \( \P \) on \( \lo \) such that
\[
\lim_{t \to +\infty} \e^{-\omega_0(\sst)t} T(t) = \P
\]
in the operator norm topology.
\end{theorem}

We concentrate on the case \( p + q \leq 1 \) and study convergence of the
Rotenberg semigroup in strong topology as \(t \to \infty \), as opposed to the
operator norm topology spoken of in Boualouar's result.

Let us begin by recalling some classical notions for Markov and substochastic
semigroups, see for example~\cite{lasota} or~\cite{pichrud}.

Let \( (X,\cl X,m) \) be a~\( \sigma \)-finite measure space and consider the
space \( L^1(X,\cl X,m) = L^1(X) \) of (equivalence classes of) absolutely
integrable functions on \( X \) with respect to \( m \).
A~linear operator \( S\colon L^1(X) \to L^1(X) \) is called \emph{substochastic}
if \( S \) is a~positive contraction, that is, given \( f \in L^1(X) \) we have
\( Sf \geq 0 \) if \( f \geq 0 \), and
\( \norm{ Sf }_{L^1(X)} \leq \norm{ f }_{L^1(X)} \), where
\( \norm{\, \cdot \,}_{L^1(X)} \) is the standard \( L^1 \)-norm.
Moreover, if \( SD_{L^1(X)} \subseteq D_{L^1(X)} \), where \( D_{L^1(X)} \) is
the set of densities in \( L^1(X) \), that is,
\[
D_{L^1(X)} \coloneqq \lrc{ f \in L^1(X)\colon f \geq 0, \, \norm{ f }_{L^1(X)} =
  1 },
\]
then \( S \) is called a~\emph{Markov} (or \emph{stochastic}) operator.
If a~substochastic operator \( S \) can be written in the form
\begin{equation*}
Sf(x) = \int_X h(x,y) f(y) m(\dd y) + Pf(x), \qquad x \in X,\ f \in L^1(X),
\end{equation*}
where \( h\colon X \times X \to \R \) is a~measurable nonnegative function
satisfying
\[
\int_X \int_X h(x,y)\, m(\ddd y)\, m(\ddd x) > 0,
\]
and \( P \) is a~positive linear operator on \( L^1(X) \), then \( S \) is
called \emph{partially integral}.
A~strongly continuous semigroup \( \sss = \sems \) in \( L^1(X) \) is said to be
\emph{substochastic} (resp.~\emph{Markov}) if \( S(t) \) is a~substochastic
(resp.~Markov) operator for every \( t \geq 0 \).
Furthermore, \( \sss \) is \emph{partially integral} if there exists
\( t_0 > 0 \) such that \( S(t_0) \) is partially integral.
If \( f_{*} \in D_{L^1(X)} \) and
\[
S(t)f_{*} = f_{*}
\]
for all \( t \geq 0 \), then \( f_{*} \) is called
an \emph{invariant density} for the semigroup.
Finally, if \( f_{*} \) is an invariant density for \( \sss \) and for all
\( f \in D_{L^1(X)} \) we have
\[
\lim_{t \to +\infty} \norm{ S(t)f - f_{*} }_{L^1(X)} = 0,
\]
then \( \sss \) is said to be \emph{asymptotically stable}.

We use the following result of Pichór and Rudnicki~\cite[Proposition
2]{pichrud}.

\begin{theorem}
\label{thm:pich}
Let \( \sss = \sems \) be a~partially integral substochastic semigroup in
\( L^1(X) \).
Assume that there is a~unique, up to an equivalence class, invariant density for
\( \sss \).
If the invariant density is almost everywhere strictly positive, then the
semigroup \( \sss \) is asymptotically stable.
\end{theorem}

Applying Theorem~\ref{thm:pich} to the Rotenberg semigroup \( \sst \), we will
obtain the main result of this paper, that is, Theorem \ref{thm:asymptotic}.
First, we state the crucial assumptions.
Let \( K \) be the operator in \( L^1(V) \) defined by
\begin{equation}
\label{eq:kernel-op}
Kf(v) \coloneqq \int_V \tm(w,v) f(w) \dd w, \qquad v \in V,\ f \in L^1(V).
\end{equation}
A~function \( \invm \in D_{L^1(V)} \) is called a \emph{stationary density} for
the kernel \( \tm \) if
\[
K\invm = \invm.
\]
We assume the following.
\begin{enumerateh}[%
resume*=h]
  \item\label{item:density} There exists a~unique, up to an equivalence class,
stationary density \( \invm \) for the kernel \( \tm \), and \( \invm \) is
strictly positive almost everywhere on \( V \).
  \item\label{item:integrable} The function
\( V \ni v \mapsto v^{-1} \invm(v) \) belongs to \( L^1(V) \).
\end{enumerateh}

\begin{theorem}
\label{thm:asymptotic}
Suppose that \( p > 0 \) and \( p+q = 1 \).
If conditions~\ref{item:density}--\ref{item:integrable} hold, then the Rotenberg
semigroup \( \sst \) is asymptotically stable with a~unique, up to an
equivalence class, invariant density \( f_{*} \) given by
\[
f_{*}(x,v) = \frac{v^{-1} \invm(v)}{\int_V w^{-1} \invm(w) \dd w }, \qquad (x,v)
\in \Omega,
\]
where \( \invm \) is the stationary density for \( \tm \).
\end{theorem}

\begin{remark}
\label{rem:a0}
If \( a > 0 \), then the function \( V \ni v \mapsto v^{-1} \) is bounded, and
in Theorem~\ref{thm:asymptotic} we may omit condition~\ref{item:integrable}.
\end{remark}

In order to prove Theorem~\ref{thm:asymptotic} we need a~couple of lemmas.

\begin{lemma}
\label{lem:gestosci-sa-stale}
If \( f_{*} \in \lo \)  is an invariant density for the Rotenberg semigroup, then
for almost every \( v \in V \) the function \( I \ni x \mapsto f_{*}(x,v) \) is
constant.
\end{lemma}

\begin{proof}
Assume that \( f_{*} \) is an invariant density for \( \sst \).
Then \( T(t)f_{*} - f_{*} = 0 \) for \( t \geq 0 \), hence \( f_{*} \in D(A) \)
and \( A f_{*} = 0 \), where \( A \) is the generator of \( \sst \) given
by~\eqref{eq:gen}.
Therefore for almost every \( v \in V \) the function
\( I \ni x \mapsto f_{*}(x,v) \) is weakly differentiable and
\( \partial_x f_{*} = 0 \).
This completes the proof.
\end{proof}

\begin{lemma}
\label{lem:rownowaznosc-gestosci}
Let \( p+q = 1 \) and assume that \( \invm \in D_{L^1(V)} \) is a~unique, up to
an equivalence class, stationary density for the kernel \( \tm \).
If \( \invm \) satisfies condition~\ref{item:integrable}, then there exists
a~unique, up to an equivalence class, invariant density \( f_{*} \in D_{\lo} \)
for the Rotenberg semigroup and
\begin{equation}
\label{eq:zaleznosc-gestosci}
f_{*}(x,v) = \frac{F_{\diamond}(v)}{\norm{ F_{\diamond} }_{L^1(V)}}
\end{equation}
for almost every \( (x,v) \in \Omega \), where \( F_{\diamond} \in L^1(V) \),
\begin{equation*}
F_{\diamond}(v) \coloneqq v^{-1} \invm(v), \qquad v \in V.
\end{equation*}
\end{lemma}

\begin{proof}
First we prove that \( f_{*} \) defined by~\eqref{eq:zaleznosc-gestosci} is
indeed an invariant density for \( \sst \).
By~\eqref{eq:mes} and since \( K\invm = \invm \) we have
\[
\int_V F_{\diamond}(w) \mes(\dd w,v) = pv^{-1} K \invm(v) + q
F_{\diamond}(v) = F_{\diamond}(v)
\]
for almost every \( v \in V \).
Hence, for \( t \in [0,b^{-1}] \) by~\eqref{eq:T-small-t} we obtain
\begin{equation*}
\norm{ F_{\diamond} }_{L^1(V)} T(t) f_{*}(x,v) = F_{\diamond}(v)[x > tv] +
\int_V F_{\diamond}(w) \mes(\dd w,v) [x \leq tv] = F_{\diamond}(v)
\end{equation*}
for almost every \( (x,v) \in \Omega \).
In other words \( T(t)f_{*} = f_{*} \) for \( t \in [0,b^{-1}] \).
If \( t > b^{-1} \), then we find a~positive integer \( n \) such that
\( s \coloneqq t/n \leq b^{-1} \).
Then \( T(s) f_{*} = f_{*} \), and by the semigroup property
\( T(t) f_{*} = [T(s)]^n f_{*} = f_{*} \).
Therefore \( f_{*} \) is an invariant density for \( \sst \).

For the uniqueness part, assume that \( f_{*}^1 \) and \( f_{*}^2 \) are
invariant densities for \( \sst \).
Let \( i \in \set{1,2} \).
By~\eqref{eq:T-small-t}, equality \( T(t)f_{*}^i = f_{*}^i \), \( t \geq 0 \)
implies that
\begin{equation*}
f_{*}^i(x,v) = f_{*}^i(x-tv,v)[x > tv] + \int_V f_{*}^i(x_{wvt},v) \mes(\dd w,v)
[x \leq tv]
\end{equation*}
for \( t \in [0,b^{-1}] \) and almost every \( (x,v) \in \Omega \).
Using Lemma~\ref{lem:gestosci-sa-stale}, this is true if and only if
\( f_{*}^i(x,v) = \int_V f_{*}^i(x,w) \mes(\dd w,v) \) for almost every
\( (x,v) \in \Omega \), which by~\eqref{eq:mes} we may rewrite as
\begin{equation*}
f_{*}^i(x,v) = p\int_V wv^{-1} \tm(w,v) f_{*}^i(x,w) \nuw + q f_{*}^i(x,v).
\end{equation*}
Since \( 1 - q = p \), this is equivalent to
\begin{equation}
\label{eq:invm-cond}
vf_{*}^i(x,v) = \int_V \tm(w,v) w f_{*}^i(x,w) \dd w.
\end{equation}
Let \( F_{*}^i \in \lo \) be given by
\[
F_{*}^i(x,v) \coloneqq v f_{*}^i(x,v), \qquad (x,v) \in \Omega,
\]
and define \( \invm^i \in L^1(V) \) by
\[
\invm^i(v) \coloneqq \frac{F_{*}^i(x,v)}{\normo{ F_{*}^i }}, \qquad (x,v) \in
\Omega;
\]
this definition make sense since by Lemma~\ref{lem:gestosci-sa-stale},
\( f_{*}^i \) does not depend on \( x \), and thus the same is true for
\( F_{*}^i \).
By~\eqref{eq:invm-cond}, \( \invm^i \) is a~stationary density for the kernel
\( \tm \).
Hence, by uniqueness assumption, we have
\begin{equation*}
\frac{F_{*}^1}{\norm{ F_{*}^1 }_{\lo}} = \frac{F_{*}^2}{\norm{ F_{*}^2}_{\lo}}.
\end{equation*}
Then
\[
\norm{ F_{*}^2 }_{\lo} f_{*}^1 = \norm{ F_{*}^1 }_{\lo} f_{*}^2.
\]
Integrating this relation over \( \Omega \), we obtain
\( \norm{ F_{*}^1 }_{L^1(V)} = \norm{ F_{*}^2 }_{L^1(V)} \), since
\( \norm{ f_{*}^i }_{\lo} = 1 \). Thus finally \( f_{*}^1 = f_{*}^2 \).
\end{proof}

\begin{lemma}
\label{lem:partially-integral}
Assume that \( p > 0 \).
The operator \( T(2b^{-1}) \) is partially integral provided that
\begin{equation}
\label{eq:calkowosc}
\int_V \int_{\max(a,b/2)}^b \tm(w,v) \dd v \dd w > 0.
\end{equation}
\end{lemma}

\begin{proof}
To simplify notation let, as before, \( s \coloneqq b^{-1} \) and
\( t \coloneqq 2s \).
By~\eqref{eq:T-small-t},
\[
T(s) = Q_1 + P_1
\]
where \( P_1 \) is a bounded positive operator in \( \lo \) and
\[
Q_1 f(x,v) = p \int_V wv^{-1}k(w,v) f(x_{wvs}, w) \dd w [x < sv], \qquad (x,v)
\in \Omega,\, f \in \lo.
\] 
Operator \( Q_1 \) describes a subpopulation of \( (x,v) \)-type cells, with
\( x < sv \), alive at time \( b^{-1} \), which are non-degenerate
daughters of cells from generation \( 0 \) (see definition~\eqref{eq:xwv}).
It follows that
\[
T(t) = Q + P
\]
where \( Q = (Q_1)^2 \) and \( P \) is another bounded positive operator in
\( \lo \).
Since
\[
(x_{wvs})_{uws} = 1 + \lrp{ 1 + xwv^{-1} - sw } uw^{-1} - su = x_{uvt} +
uw^{-1},
\]
we have
\begin{equation}
\label{eq:czesc-calkowaprim}
Q f(x,v) = p^2 \int_V \int_V uv^{-1} k(u,w) k(w,v) f\lrp{ x_{uvt} + uw^{-1},
  u} \lrb{ x_{wvs} < sw }
\lrb{x < sv} \dd u \dd w
\end{equation}
for \( (x,v) \in \Omega \).
Furthermore, a~little bit of algebra shows that if \( x_{wvs} < sw \) for some
\( w \in V \), which is the same as \( x_{wvt} < 0 \), then \( x < sv \).
Thus, \( [ x_{wvs} < sw ] [x < sv] \) in the formula above may be replaced by
\( [x_{wvt} < 0] \).
Applying the Fubini theorem and substituting
\begin{equation*}
y = y(w) 
\coloneqq x_{uvt} + uw^{-1},
\end{equation*}
in~\eqref{eq:czesc-calkowaprim} we get
\begin{equation}
\label{eq:calkowa-podst}
Qf(x,v) = p^2 \int_V \int_{x_{uvt}+ub^{-1}}^{x_{uvt}+ua^{-1}} \tiy^2 v^{-1}
\tm(\tiy,v) \tm(u,\tiy) f(y,u) [x_{\tiy vt} < 0] \dd y \dd u,
\end{equation}
where \( \tiy = u(y-x_{uvt})^{-1} \). 
If \( a = 0 \), then by convention we take \( a^{-1} \coloneqq +\infty \) here.
For \( x \in I \) and \( u,v \in V \), we have
\[
(x_{uvt} + ub^{-1}, x_{uvt}+ua^{-1}) \cap \lrc{ y \in \R\colon x_{\tiy vt} < 0 }
\subseteq I.
\]
Indeed, \( x_{uvt} + ub^{-1} = x_{uvb^{-1}} > 1 - ub^{-1} > 0 \), and if
\( x_{\tiy vt} < 0 \), then \( \tiy(t-xv^{-1}) > 1 \) which implies
\( y < u(t-xv^{-1})+x_{uvt} = 1 \).
Therefore, we may rewrite~\eqref{eq:calkowa-podst} as
\begin{equation*}
Qf(x,v) = \int_V \int_I h(x,v,y,u) f(y,u) \dd y \dd u,
\end{equation*}
for
\[
h(x,v,y,u) \coloneqq p^2 \tiy^2 v^{-1} \tm(\tiy,v) \tm(u,\tiy) [y \in
\Lambda_{uvt}][x_{\tiy vt} < 0] \geq 0,
\]
where \( (x,v), (y,u) \in \Omega \) and
\( \Lambda_{uvt} = (x_{uvt}+ub^{-1},x_{uvt}+ua^{-1}) \).
Hence, we are left with proving that
\( \int_{\Omega} \int_{\Omega} h \dd \mu \dd \mu > 0 \).
We have
\[
\int_{\Omega} \int_{\Omega} h \dd \mu \dd \mu = \int_{\Omega} Q \ind_\Omega \dd
\mu.
\]
Let
\[
d \coloneqq \max\lrp{a,t^{-1}} = \max(a, b/2)
\]
and note that \( x \in I \) and \( x_{wvt} < 0 \) is equivalent to
\( d < w < b \) and \( 0 < x < tv - vw^{-1} \).
Thus by~\eqref{eq:czesc-calkowaprim} we obtain (recall that
\( \lrb{ x_{wvs} < sw } \lrb{x < sv } \) may be replaced by \( [x_{wvt} < 0] \))
\begin{align*}
  \frac{1}{p^2}\int_{\Omega} \int_{\Omega} h \dd \mu \dd \mu
  &= \int_V \int_I \int_V \int_V uv^{-1}
    \tm(u,w) \tm(w,v) [x_{wvt} < 0] \dd u \dd w \dd x \dd v\\
  &= \int_V \int_d^b \int_V \int_0^{tv-vw^{-1}} uv^{-1}
    \tm(u,w) \tm(w,v) \dd x \dd v \dd w \dd u\\
  &= \int_V \int_d^b \int_V u \tm(u,w) \tm(w,v)
    \lrp{t-w^{-1}} \dd v \dd w
    \dd u\\
  &= \int_V \int_d^b u \tm(u,w) \lrp{t-w^{-1}} \dd w \dd u,
\end{align*}
the second equality resulting from the Fubini theorem, and the third
from~\eqref{eq:rozklad-k}.
Since the function \( V \times \lrp{d,b} \ni (u,w) \mapsto u(t-w^{-1}) \) is
strictly positive, the integral
\( \int_V \int_d^b u \tm(u,w) (t-w^{-1}) \dd w \dd u \) is positive if and only
if condition~\eqref{eq:calkowosc} holds.
\end{proof}

\begin{lemma}
\label{lem:getosc-calkowosc}
If there exits a~strictly positive stationary density \( \invm \in D_{L^1(V)} \)
for the kernel \( \tm \), then
\begin{equation*}
\int_V \int_{U} \tm(w,v) \dd v \dd w > 0
\end{equation*}
for every Lebesgue measurable subset \( U \) of \( V \) with positive measure.
\end{lemma}

\begin{proof}
Assume, contrary to our claim, that
\begin{equation}
\label{eq:zero}
\int_V \int_U \tm(w,v) \dd v \dd w = 0.
\end{equation}
for a set \( U \) of positive measure.
However, for almost every \( v \in V \) we have
\begin{equation*}
\invm(v) = \int_V \tm(w,v) \invm(w) \dd w.
\end{equation*}
Integrating this over \( U \),
\begin{equation}
\label{eq:calka-z-gestosci}
\int_U \invm(v) \dd v = \int_V \int_{U} \tm(w,v) \invm(w) \dd v \dd
w
\end{equation}
by the Fubini theorem.
Since \( \tm \) is nonnegative, \eqref{eq:zero} implies \( \tm = 0 \) almost
everywhere on \( V \times U \).
This means that
\[
\int_{U} \invm(v) \dd v = 0
\]
by~\eqref{eq:calka-z-gestosci}, which contradicts strict positivity of
\( \invm \).
This completes the proof.
\end{proof}

Now we are ready to prove the main theorem.

\begin{proof}[Proof of Theorem~\textup{\ref{thm:asymptotic}}]
We know by Theorem~\ref{thm:norma-polgrupy}~\ref{item:5} that the Rotenberg
semigroup \( \sst \) is Markov, hence in particular substochastic.
By Lemma~\ref{lem:rownowaznosc-gestosci} there exists a~unique invariant density
\( f_{*} \) for \( \sst \).
Moreover, by~\eqref{eq:zaleznosc-gestosci}, \( f_{*} \) is strictly positive
since so is \( \invm \).
Furthermore, Lemma~\ref{lem:getosc-calkowosc} with
\( U = \lrp[\big]{ \max(a,b/2), b } \) implies that \( \sst \) is partially
integral by Lemma~\ref{lem:partially-integral}.
Thus, we may apply Theorem~\ref{thm:pich} which completes the proof.
\end{proof}

With the described approach asymptotic stability of the Rotenberg semigroup is
more difficult to prove if \( p + q > 1 \).
In this case the semigroup \( \sst \) needs not be substochastic and we should
first rescale it, taking
\[
S(t) \coloneqq \e^{-\omega_0(\sst) t} T(t), \qquad t \geq 0.
\]
Then the semigroup \( \sss = \sems \) is bounded and its generator equals
\( A - \omega_0(\sst) \), where \( A \) is the generator of \( \sst \).
Arguing as in the proofs of Lemma~\ref{lem:gestosci-sa-stale} and
Lemma~\ref{lem:rownowaznosc-gestosci}, \( f_{*} \in D_{\lo} \) is an invariant
density for \( \sss \) if and only if
\[
v \lrp[\big]{1 - q \e^{-\omega_0(\sst)v^{-1}}} f_{*}(x,v) = p \int_V w
\tm(w,v) \e^{-\omega_0(\sst)w^{-1}} f_{*}(x,w) \dd w
\]
for almost every \( (x,v) \in \Omega \).
The problem here is that we do not have explicit formula for
\( \omega_0(\sst) \) what we essentially discussed at the end of
Section~\ref{sec:growth-estimates}.

Finally, let us state a~result describing the asymptotic behaviour of the
Rotenberg semigroup in the case \( p + q < 1 \).

\begin{theorem}
\label{thm:asymp-spadek}
Suppose that \( p + q < 1 \).
\begin{enumerate}
  \item\label{item:mocno} If \( a = 0 \), then \( \sst \) converges strongly to
zero, that is,
\[
\lim_{t \to +\infty} T(t)f = 0, \qquad f \in \lo.
\]
  \item\label{item:normowo} If \( a > 0 \), then \( \sst \) converges to zero in
the operator norm, that is,
\[
\lim_{t \to +\infty} \norm{ T(t) }_{\lino\lo} = 0.
\]
\end{enumerate}
\end{theorem}

Note that in the case \( p + q < 1 \) the cell population gradually dies out.
Theorem~\ref{thm:norma-polgrupy-leb}~\ref{item:3} reflect the fact that in the
case \( a = 0 \) for every \( t \geq 0 \) there may exist cells that mature very
slowly and did not yet divide.
In the case \( a > 0 \) all parts of the population die out uniformly fast.

\begin{proof}
Part~\ref{item:normowo} follows directly by
Theorem~\ref{thm:norma-polgrupy-leb}~\ref{item:4}, and hence we assume that
\( a = 0 \).
Since \( \norm{ T(t) }_{\lino\lo} \leq 1 \) for each \( t \geq 0 \) by
Theorem~\ref{thm:norma-polgrupy}~\ref{item:grube}, it is sufficient to show that
\begin{equation}
\label{eq:mocny-ciag}
\lim_{n \to +\infty} T \lrp{ nb^{-1} } f = 0, \qquad f \in \lo.
\end{equation}

For \( n \geq 1 \) and \( f \in \lo \) by~\eqref{eq:wzor-na-polgrupe} we have
\begin{equation}
\label{eq:norma-polgrupy-mala}
\begin{split}
\norm{ T \lrp{ nb^{-1} }f }_{\lo} &= \int_V \int_I \abs{
  \tif(x-nvb^{-1},v) } \dd x \dd v\\
&= \int_V \int_{-nvb^{-1}}^{1-nvb^{-1}} \abs{ \tif(x,v) } \dd x \dd v.
\end{split}
\end{equation}
However, by~\eqref{eq:norma-ext-ogr} with \( \omega = 0 \),
\[
\int_{\tio} \abs{ \tif } \dd \mu < +\infty,
\]
provided \( p + q < 1 \).
Thus~\eqref{eq:norma-polgrupy-mala} and the Lebesgue dominated convergence
theorem imply~\eqref{eq:mocny-ciag}.
\end{proof}

\begin{remark}
The proof of Theorem~\ref{thm:asymp-spadek}~\ref{item:mocno} still works in the
case where \( (V,\nu) \) is any measure space with \( V \subseteq (a,b) \) and
\( a \geq 0 \).
This is obvious since~\eqref{eq:norma-ext-ogr} holds in this general setup.
Hence, if \( p + q < 1 \), then the Rotenberg semigroup converges strongly to
zero even if we do not assume that \( V = (a,b) \) with the Lebesgue measure.
\end{remark}

\section{Discussion of assumptions}
\label{sec:disc-assumpt}

In our last theorem we discuss the relation between
conditions~\ref{item:irr}--\ref{item:bounded} used in
Theorem~\ref{thm:boulanouara} and~\ref{item:density}--\ref{item:integrable} used
in Theorem~\ref{thm:asymptotic}.
As a consequence of this result, we see in particular that Boulanouar's
assumptions are stronger than our condition~\ref{item:density}.
In this context, it becomes clear that it is~\ref{item:integrable} that is
crucial for our analysis.
As we have seen, \ref{item:integrable} is an assumption of existence and
uniqueness of an invariant density for the Rotenberg semigroup.
See the upcoming~\cite{rudpich17} for a~way to deduce existence and uniqueness
of an invariant density from properties of a semigroup.

\begin{theorem}
\label{thm:slabszy-warunek}
For the kernel \( \tm \) the following is true.
\begin{enumerate}
  \item\label{item:b-mocniejszy} If
conditions~\ref{item:irr}--\ref{item:bounded} hold, then so
does~\ref{item:density}, but not necessarily~\ref{item:integrable}.
  \item\label{item:b-slabszy} If condition~\ref{item:density} holds, then so
does~\ref{item:irr}, but not necessarily~\ref{item:bounded} even if we
additionally assume~\ref{item:integrable}.
\end{enumerate}
\end{theorem}

\begin{proof}
For part~\ref{item:b-mocniejszy}, in order to show that~\ref{item:irr}
and~\ref{item:bounded} implies~\ref{item:density} we first note that the
operator \( K \) defined by~\eqref{eq:kernel-op} is weakly compact
provided~\ref{item:irr} and~\ref{item:bounded}.
Indeed, let \( S_V \) be the closed unit sphere in \( L^1(V) \).
We will show that \( KS_V \) is relatively weakly compact, that is, the weak
closure of \( KS_V \) is compact in the weak topology of \( L^1(V) \).
By the Dunford-Pettis theorem~\cite[Theorem 5.2.9]{albiac} this holds if and
only if \( KS_V \) is composed of uniformly integrable functions.
Let \( U \) be a~measurable subset of \( V \).
For \( f \in S_V \) by~\ref{item:bounded} we have
\[
\int_{U} \abs{ Kf(v) } \dd v \leq \norm{ k }_{L^{\infty}(V \times V)} \int_{U}
\int_V \abs{ f(w) } \dd w \dd v = \norm{ k }_{L^{\infty}(V \times V)} \leb(U),
\]
where \( \norm{ k }_{L^{\infty}(V \times V)} \) is the essential bound of
\( k \) on \( V \times V \).
This implies uniform integrability of functions belonging to \( KS_V \), and
completes the proof of weak compactness of \( K \).
Hence \( K^2 \) is a~compact operator by \cite[Corollary~VI.8.13]{dunford}, the
result also due to Dunford and Pettis. 
Therefore, by~\ref{item:irr} and the Jentzsch theorem~\cite[V.6.6]{schaefer},
for the spectral radius \( r(K) \) of \( K \) there exists a~unique density
\( \invm \in D_{L^1(V)} \) satisfying \( K\invm = r(K) \invm \), and moreover
\( \invm > 0 \) almost everywhere on \( V \).
However, by~\eqref{eq:rozklad-k} it follows that \( r(K) = 1 \), and
thus~\ref{item:density} holds.

To see that~\ref{item:irr} and~\ref{item:bounded} do not
imply~\ref{item:integrable} we take \( V \coloneqq (0,1) \), and let \( \tm \)
be equal identically \( 1 \) on \( V \times V \).
Then~\ref{item:irr} and~\ref{item:bounded} are satisfied and \( \invm \)
equaling identically \( 1 \) on \( V \) is the stationary density for \( \tm \),
and yet \( \int_V v^{-1} \invm(v) \dd v = +\infty \).

For part~\ref{item:b-slabszy}, assume that~\ref{item:density} holds, and let
\( \invm \) be the unique stationary density for \( \tm \).
Suppose, contrary to our claim, that there exists a~measurable set
\( U \subset V \), such that
\( \leb(U), \leb(V \setminus U) > 0 \) and
\begin{equation}
\label{eq:zerok}
\int_{V \setminus U} \int_{U} \tm(w,v) \dd v \dd w = 0.
\end{equation}
Let \( f \in L^1(V) \) be defined as \( f \coloneqq \ind_{U} \invm \),
where \( \ind_{U} \) is the indicator function of \( U \).
Then
\begin{equation}
\label{eq:kff}
Kf \leq K\invm = \invm
\end{equation}
almost everywhere on \( V \), since \( K \) is a~positive operator.
Moreover
\[
Kf(v) = \int_V \tm(w,v) f(w) \dd w = \int_{U} \tm(w,v) \invm(w) \dd w,
\]
and hence
\[
\int_{V \setminus U} Kf(v) \dd v = 0,
\]
since by~\eqref{eq:zerok} we have \( \tm = 0 \) almost everywhere on
\( (V \setminus U) \times U \).
Because \( Kf \) is nonnegative, this means that \( Kf = 0 \) almost everywhere
on \( V \setminus U \).
Thus, by~\eqref{eq:kff}, we obtain \( Kf \leq f \) almost everywhere on \( V \).
However, by~\eqref{eq:rozklad-k} and the Fubini theorem,
\[
\norm{ Kf }_{L^1(V)} = \norm{ f }_{L^1(V)},
\]
which implies \( Kf = f \) almost everywhere on \( V \).
Then \( g \coloneqq f/\norm{ f }_{L^1(V)} \) is a~stationary density for
\( \tm \), and \( g \neq \invm \), which contradicts~\ref{item:density}.
This contradiction proves that~\ref{item:density} implies~\ref{item:irr}.

Finally, to show that~\ref{item:density} and~\ref{item:integrable} do not
imply~\ref{item:bounded}, let \( V \coloneqq (0,1) \), and
\[
\tm(w,v) \coloneqq \frac{3}{4} \cdot \frac{v}{\sqrt{1-v}}, \qquad w,v \in V.
\]
Then \( \tm \) satisfies~\eqref{eq:rozklad-k}.
Moreover, \( \invm \in L^1(V) \) defined by
\[
\invm(v) \coloneqq \frac{3}{4} \cdot \frac{v}{\sqrt{1-v}}, \qquad v \in V,
\]
is a~stationary density for \( \tm \).
Next, if \( f \) is a~stationary density for \( \tm \), then
\[
f(v) = \int_V \tm(w,v) f(w) \dd w = \frac{3}{4} \cdot \frac{v}{\sqrt{1-v}}
\int_V f(w) \dd w = \frac{3}{4} \cdot \frac{v}{\sqrt{1-v}}, \qquad v \in V.
\]
Hence, \( \invm \) is a unique stationary density for \( \tm \), and it is
strictly positive, that is, \ref{item:density} holds for \( \tm \).
Furthermore, it is easy to check that~\ref{item:integrable} is satisfied,
however \( k \) is not essentially bounded and~\ref{item:bounded} fails to hold.
\end{proof}

\section*{Acknowledgements}

I~would like to thank Professors A.~Bobrowski, J.~Banasiak and R.~Rudnicki whose
remarks helped to improve the paper.

\bibliographystyle{amsplain}
\bibliography{references}

\end{document}